\documentclass[10pt]{article}

\usepackage[margin=1in]{geometry}
\usepackage{amsmath,amssymb,amsthm,mathtools,mathrsfs}
\usepackage{enumitem}
\usepackage{aliascnt}
\usepackage[colorlinks=true,linkcolor=blue,citecolor=blue,urlcolor=blue]{hyperref}
\usepackage[nameinlink,capitalise,noabbrev]{cleveref}
\usepackage{microtype}

\numberwithin{equation}{section}
\setlist[enumerate]{leftmargin=2.2em,itemsep=0.2em,topsep=0.35em}

\newtheorem{theorem}{Theorem}[section]
\newaliascnt{proposition}{theorem}
\newtheorem{proposition}[proposition]{Proposition}
\aliascntresetthe{proposition}
\newaliascnt{lemma}{theorem}
\newtheorem{lemma}[lemma]{Lemma}
\aliascntresetthe{lemma}
\newaliascnt{corollary}{theorem}
\newtheorem{corollary}[corollary]{Corollary}
\aliascntresetthe{corollary}
\theoremstyle{remark}
\newaliascnt{remark}{theorem}
\newtheorem{remark}[remark]{Remark}
\aliascntresetthe{remark}

\crefname{theorem}{Theorem}{Theorems}
\Crefname{theorem}{Theorem}{Theorems}
\crefname{proposition}{Proposition}{Propositions}
\Crefname{proposition}{Proposition}{Propositions}
\crefname{lemma}{Lemma}{Lemmas}
\Crefname{lemma}{Lemma}{Lemmas}
\crefname{corollary}{Corollary}{Corollaries}
\Crefname{corollary}{Corollary}{Corollaries}
\crefname{remark}{Remark}{Remarks}
\Crefname{remark}{Remark}{Remarks}

\newcommand{\R}{\mathbb R}
\newcommand{\N}{\mathbb N}
\newcommand{\Z}{\mathbb Z}
\newcommand{\Q}{\mathbb Q}
\newcommand{\C}{\mathbb C}
\newcommand{\E}{\mathcal E}
\newcommand{\Prob}{\mathbb P}
\newcommand{\Ex}{\mathbb E}
\newcommand{\Borel}{\mathcal B}

\newcommand{\Span}{\operatorname{span}}
\newcommand{\supp}{\operatorname{supp}}
\newcommand{\Spec}{\operatorname{Spec}}
\newcommand{\Disc}{\operatorname{Disc}}
\newcommand{\Tr}{\operatorname{Tr}}
\newcommand{\Mult}{\operatorname{Mult}}
\newcommand{\Dom}{\operatorname{Dom}}
\newcommand{\dist}{\operatorname{dist}}

\hypersetup{
 pdftitle={Spectral Simplicity and Joint Eigenvalue Densities for a Non-Gaussian Brownian Time Change},
 pdfauthor={Chunhao Cai},
 pdfsubject={Spectral response for a non-Gaussian Brownian time change},
 pdfkeywords={non-Gaussian multiplicative chaos, spectral response,
 Brownian time change, simple spectrum, joint eigenvalue density}
}

\title{Spectral Simplicity and Joint Eigenvalue Densities\\
for a Non-Gaussian Brownian Time Change}
\author{Chunhao Cai\\
\small School of Mathematics (Zhuhai), Sun Yat-sen University\\
\small \texttt{caichh9@mail.sysu.edu.cn}}
\date{}

\begin{document}
\maketitle

\begin{abstract}
We study a Brownian time change on the unit square whose speed measure is
constructed from a Dirichlet eigenfunction expansion with independent
non-Gaussian coefficients.  For $0<\gamma<\sqrt2$, the measure is obtained by
a second-moment martingale argument.  Finite coefficient translations induce
coherent exponential tilts of the speed measure, and conditioning on the
complementary coefficients gives positive Lebesgue densities on every
finite-dimensional orbit.  Unitary transport along these orbits gives a
common-domain analytic family and explicit first-order cluster derivatives.
A first-order splitting argument proves almost-sure simplicity, while the local
eigenfunction-square identity and a Vandermonde argument give joint densities
for all finite vectors of ordered positive eigenvalues.
\end{abstract}



\begingroup
\small
\tableofcontents
\endgroup

\section{Introduction}

Multiplicative chaos originates in Kahane's construction
\cite{Kahane1985} and is by now especially well developed in the Gaussian
log-correlated setting.  Non-Gaussian analogues have also appeared in several
forms.  Log-infinitely divisible multifractal measures and their star-scale
invariant extensions were developed in \cite{BacryMuzy2003,RhodesSohierVargas2014}.
For log-correlated fields built as sums of independent non-Gaussian fields,
Junnila \cite{Junnila2020} proved convergence, moment estimates, and
analyticity of the associated chaos throughout the subcritical regime; a
basic example is the random Fourier series with independent identically
distributed non-Gaussian coefficients.  More recently, Kim and Kriechbaum
\cite{KimKriechbaum} constructed couplings under which the Gaussian and
non-Gaussian Fourier chaoses are mutually absolutely continuous in part of
the subcritical regime, and Basu Roy Chowdhury and Ganguly
\cite{BasuGanguly} extended this invariance principle to the full subcritical
regime.  A different genuinely non-Gaussian family, built from planar
Brownian local times, is provided by Brownian multiplicative chaos
\cite{Jego2020,Jego2021}.

On the diffusion side, Gaussian Liouville measure gives rise to Liouville
Brownian motion, constructed in \cite{GarbanRhodesVargas2016}; its Dirichlet
form and heat kernel were further studied in
\cite{GarbanRhodesVargas2014,AndresKajino2016}.  Berestycki and Wong
\cite{BerestyckiWong2023} established a Weyl law for its eigenvalues.  The
response calculus of \cite{CaiResponseSubmitted} isolates a pathwise
mechanism for spectral splitting and joint eigenvalue densities along
finite-dimensional exponential tilts; in the Gaussian setting,
finite-dimensional Cameron--Martin coordinates provide the corresponding
slicing law.

The present paper shows that the spectral mechanism does not require
Gaussian coefficients.  For the product field considered below, conditioning
finitely many coefficients replaces Cameron--Martin slicing by conditional
laws with positive Lebesgue densities.  Gaussian comparison is used only to
transfer the trace properties of the limiting speed measure.  To the best of
our knowledge, almost-sure spectral simplicity has not previously been
established for a Brownian time change driven by a genuinely non-Gaussian
multiplicative-chaos speed measure.

Let
\[
 D=(0,1)^2,
 \qquad
 e_{m,\ell}(x_1,x_2)=2\sin(m\pi x_1)\sin(\ell\pi x_2),
 \qquad
 \lambda_{m,\ell}=\pi^2(m^2+\ell^2),
\]
and enumerate these Dirichlet eigenpairs as $(e_k,\lambda_k)_{k\ge1}$,
repeating eigenvalues according to multiplicity.  Let $\mathsf p$ be a
non-Gaussian probability law on $\R$ with an everywhere strictly positive
Borel density $p:\R\to(0,\infty)$ satisfying
\begin{equation}
 \int_{\R}x p(x)\,dx=0,
 \qquad
 \int_{\R}x^2p(x)\,dx=1,
 \qquad
 \int_{\R}e^{tx}p(x)\,dx<\infty\quad(t\in\R).
 \label{eq:coefficient-law}
\end{equation}
On
\[
 \Omega=\R^{\N},
 \qquad
 b=(b_k)_{k\ge1},
 \qquad
 \Prob=\mathsf p^{\otimes\N},
\]
define
\begin{equation}
 S_{n,b}(x)
 :=\sqrt{2\pi}\sum_{\lambda_k\le2^{2n}}
 b_k\lambda_k^{-1/2}e_k(x)
 \label{eq:field}
\end{equation}
and
\begin{equation}
 M_{n,b}(dx)
 :=\frac{e^{\gamma S_{n,b}(x)}}{Z_{n,\gamma}(x)}\,dx,
 \qquad
 Z_{n,\gamma}(x):=\Ex e^{\gamma S_{n,b}(x)}.
 \label{eq:chaos-approx}
\end{equation}
Fix $0<\gamma<\sqrt2$ throughout.
Section~\ref{sec:measure} constructs a canonical almost-sure limit $M_b$ of
\eqref{eq:chaos-approx}.  On a measurable set of full probability it is a finite,
non-zero, atomless Radon measure with full support, is singular with respect
to Lebesgue measure, and is a smooth measure of full quasi-support for the
Dirichlet form
\begin{equation}
 \E(u,v):=\frac1{2\pi}\int_D\nabla u\cdot\nabla v\,dx,
 \qquad u,v\in H_0^1(D).
 \label{eq:energy}
\end{equation}
We then set
\begin{equation*}
 V_b:=\{u\in H_0^1(D):\widetilde u\in L^2(M_b)\},
\end{equation*}
where $\widetilde u$ is a fixed quasi-continuous representative, and denote by
$A_b$ the non-negative self-adjoint operator associated with
$(\E,V_b)$ on $L^2(M_b)$.

\begin{theorem}
\label{thm:main}
For every coefficient law satisfying \eqref{eq:coefficient-law} and every
$0<\gamma<\sqrt2$, almost surely, the operator $A_b$ has compact resolvent
and its positive eigenvalues satisfy
\begin{equation*}
 0<\Lambda_1^b<\Lambda_2^b<\cdots\uparrow\infty.
\end{equation*}
Moreover, for every $N\ge1$ and every
$1\le n_1<\cdots<n_N$,
\begin{equation*}
 (\Lambda_{n_1}^b,\ldots,\Lambda_{n_N}^b)_\#\Prob
 \ll \mathcal L^N.
\end{equation*}
Here $\mathcal L^N$ denotes $N$-dimensional Lebesgue measure.
\end{theorem}

\begin{remark}
The restriction $\gamma<\sqrt2$ comes from the second-moment construction,
since the two-point kernel has order $|x-y|^{-\gamma^2}$.  The spectral
argument depends on the structural properties of the limiting measure and the
conditional orbit densities.  A full-subcritical extension would therefore
reduce to constructing the measure with a positive Frostman exponent and
transferring the trace properties; compare
\cite[Section~3.4]{BasuGanguly}, \cite[Section~3]{KimKriechbaum}, and
\cite[Section~6]{CaiResponseSubmitted}.
\end{remark}

The rest of the article is organized as follows.  Section~2 constructs the
non-Gaussian speed measure, establishes its trace-form and spectral
realization, and derives the finite-dimensional conditional coefficient
slices.  Section~3 develops the pathwise response argument, including
first-order splitting and the square-identity proof of transversality.
Section~4 proves almost-sure simplicity by one-dimensional conditional
slicing, and Section~5 proves joint eigenvalue densities by finite-dimensional
response submersion.

\section{The speed measure and conditional coefficient slices}
\label{sec:measure}

We construct the limiting measure, place its trace properties on a measurable
full-probability event, and identify the associated Green operator.  The
finite-coordinate disintegration at the end of the section supplies the
slices used in the spectral arguments.

\subsection{Spectral covariance and construction of the measure}

Put
\begin{equation*}
 \psi_k(x):=\sqrt{2\pi}\,\lambda_k^{-1/2}e_k(x),
 \qquad
 K_n(x,y):=\sum_{\lambda_k\le2^{2n}}\psi_k(x)\psi_k(y).
\end{equation*}
Thus $S_{n,b}=\sum_{\lambda_k\le2^{2n}}b_k\psi_k$ and $K_n$ is its covariance
kernel; the normalization gives unit logarithmic coefficient.

For $R\ge1$ and $z\in\R^2$, set
\begin{equation}
 T_R(z):=
 \sum_{\substack{k\in\Z^2\setminus\{0\}\\ |k|\le R}}
 \frac{\cos(\pi k\cdot z)}{|k|^2},
 \qquad
 d_2(z):=\operatorname{dist}(z,2\Z^2),
 \label{eq:lattice-green-def}
\end{equation}
and
\begin{equation*}
 L_R(z):=\log^+\!\frac1{d_2(z)\vee R^{-1}}.
\end{equation*}
Here $T_R$ is the sharply truncated periodic lattice Green sum, while $L_R$
records its logarithmic singularity at the cutoff scale $R^{-1}$.

\begin{lemma}
\label{lem:lattice-green}
There is a deterministic $C<\infty$ such that
\begin{equation}
 \bigl|T_R(z)-2\pi L_R(z)\bigr|\le C,
 \qquad R\ge1,\quad z\in\R^2.
 \label{eq:lattice-green-estimate}
\end{equation}
\end{lemma}

\begin{proof}
We first replace the sharp cutoff by a heat-smoothed one.  Set
\begin{equation*}
 \widetilde T_R(z):=
 \sum_{k\in\Z^2\setminus\{0\}}
 \frac{e^{-|k|^2/R^2}\cos(\pi k\cdot z)}{|k|^2}.
\end{equation*}
The two cutoffs differ by a uniformly bounded amount.  Indeed,
using $1-e^{-u}\le u$ and dyadic annuli,
\begin{align}
 |T_R(z)-\widetilde T_R(z)|
 &\le
 \sum_{0<|k|\le R}\frac{1-e^{-|k|^2/R^2}}{|k|^2}
 +\sum_{j\ge0}\sum_{2^jR<|k|\le2^{j+1}R}
   \frac{e^{-|k|^2/R^2}}{|k|^2}\notag\\
 &\le R^{-2}\#\{k\in\Z^2:|k|\le R\}
 +C\sum_{j\ge0}e^{-4^j}
 \le C.
 \label{eq:sharp-smooth-difference}
\end{align}
The smoothed sum has a convenient integral representation.  Namely,
\begin{equation*}
 \frac{e^{-|k|^2/R^2}}{|k|^2}
 =\int_{R^{-2}}^\infty e^{-t|k|^2}\,dt,
\end{equation*}
so Fubini gives
\begin{equation}
 \widetilde T_R(z)
 =\int_{R^{-2}}^\infty\bigl(\vartheta_t(z)-1\bigr)\,dt,
 \qquad
 \vartheta_t(z):=\sum_{k\in\Z^2}e^{-t|k|^2}e^{i\pi k\cdot z}.
 \label{eq:theta-integral}
\end{equation}
Poisson summation yields
\begin{equation}
 \vartheta_t(z)
 =\frac{\pi}{t}\sum_{m\in\Z^2}
 \exp\!\left(-\frac{\pi^2|z-2m|^2}{4t}\right).
 \label{eq:poisson-theta}
\end{equation}
We estimate the integral in \eqref{eq:theta-integral}.  Its part over
$t\ge1$ is uniformly bounded, since
\[
 |\vartheta_t(z)-1|
 \le\sum_{k\ne0}e^{-t|k|^2}\le Ce^{-t}.
\]
Consider next $R^{-2}\le t\le1$.  If $d_2(z)\ge1/2$, then
for $0<t\le1$ every image point stays a fixed positive distance from $z$,
and \eqref{eq:poisson-theta} gives
\begin{equation*}
 |\vartheta_t(z)|\le Ct^{-1}e^{-c/t},
 \qquad
 \int_{R^{-2}}^1|\vartheta_t(z)-1|\,dt\le C.
\end{equation*}
Since $L_R(z)\le\log2$, this gives the required estimate in that case.
Suppose therefore that $d_2(z)<1/2$.  There is then a unique
$m_0\in\Z^2$ with $|z-2m_0|=d_2(z)$, and
\eqref{eq:poisson-theta} becomes, uniformly for $0<t\le1$,
\begin{equation*}
 \vartheta_t(z)
 =\frac{\pi}{t}
 \exp\!\left(-\frac{\pi^2d_2(z)^2}{4t}\right)
 +O\!\left(t^{-1}e^{-c/t}\right).
\end{equation*}
The error integrates to $O(1)$; the integral of the subtracted constant
$-1$ in \eqref{eq:theta-integral} is also $O(1)$.  For the main term put
$c_0=\pi^2/4$.  If $d_2(z)\ge R^{-1}$, the substitution
$u=c_0d_2(z)^2/t$ gives
\begin{align*}
 \pi\int_{R^{-2}}^1
 t^{-1}e^{-c_0d_2(z)^2/t}\,dt
 &=\pi\int_{c_0d_2(z)^2}^{c_0d_2(z)^2R^2}
   \frac{e^{-u}}u\,du\\
 &=2\pi\log\frac1{d_2(z)}+O(1),
\end{align*}
where the error is uniform because the upper limit is at least $c_0$ and
$\int_{c_0}^{\infty}e^{-u}u^{-1}\,du<\infty$.
If $d_2(z)<R^{-1}$, then
\begin{align}
 &\left|
 \pi\int_{R^{-2}}^1t^{-1}e^{-c_0d_2(z)^2/t}\,dt
 -2\pi\log R\right|\notag\\
 &\qquad\le
 \pi\int_{R^{-2}}^1
 \frac{1-e^{-c_0d_2(z)^2/t}}{t}\,dt
 \le\pi\int_0^{c_0}\frac{1-e^{-u}}u\,du
 \le C.
 \label{eq:theta-integral-near}
\end{align}
Equations \eqref{eq:theta-integral}--\eqref{eq:theta-integral-near}
show
\[
 \widetilde T_R(z)=2\pi L_R(z)+O(1)
\]
uniformly in $R$ and $z$.  Together with
\eqref{eq:sharp-smooth-difference}, this proves
\eqref{eq:lattice-green-estimate}.
\end{proof}

The reflection formula for the tensor sine basis now transfers the lattice
estimate to the covariance kernel.

\begin{lemma}
\label{lem:spectral-covariance}
There is a deterministic $C<\infty$ such that, for every $n\ge1$ and
$x,y\in D$,
\begin{equation}
 K_n(x,y)\le C+\log^+\frac1{|x-y|}.
 \label{eq:spectral-covariance-bound}
\end{equation}
Moreover, for every compact $K\Subset D$ there is $C_K<\infty$ such that
\begin{equation}
 \left|K_n(x,y)-\log^+\frac1{|x-y|\vee2^{-n}}\right|\le C_K,
 \qquad x,y\in K,\quad n\ge1.
 \label{eq:spectral-covariance-local}
\end{equation}
\end{lemma}

\begin{proof}
Write $x=(x_1,x_2)$ and $y=(y_1,y_2)$, and set
\begin{equation*}
 a=x_1-y_1,\qquad b=x_1+y_1,\qquad
 c=x_2-y_2,\qquad d=x_2+y_2.
\end{equation*}
For $R\ge1$ set
\begin{equation*}
 A_R(u,v):=
 \sum_{\substack{m,\ell\ge1\\m^2+\ell^2\le R^2}}
 \frac{\cos(\pi m u)\cos(\pi\ell v)}{m^2+\ell^2}.
\end{equation*}
Separating the full lattice sum into quadrants and then removing the two
coordinate axes in \eqref{eq:lattice-green-def} gives
\begin{equation*}
 4A_R(u,v)
 =T_R(u,v)
 -2\sum_{1\le m\le R}\frac{\cos(\pi m u)}{m^2}
 -2\sum_{1\le\ell\le R}\frac{\cos(\pi\ell v)}{\ell^2}.
\end{equation*}
The one-dimensional sums are uniformly bounded.  Hence
\cref{lem:lattice-green} implies
\begin{equation*}
 A_R(u,v)=\frac\pi2L_R(u,v)+O(1)
\end{equation*}
uniformly in $R,u,v$, where $L_R$ is the function defined above.

The reflection identity for the tensor sine basis now gives the covariance
kernel.  Since $\lambda_{m,\ell}=\pi^2(m^2+\ell^2)$ and
$e_{m,\ell}(x)=2\sin(m\pi x_1)\sin(\ell\pi x_2)$, with
$R_n=2^n/\pi$ and the finitely many cases $R_n<1$ absorbed into the
constant, we obtain
\begin{align}
 \sum_{\lambda_k\le2^{2n}}\psi_k(x)\psi_k(y)
 &=\frac2\pi\Bigl[
 A_{R_n}(a,c)-A_{R_n}(a,d)
 -A_{R_n}(b,c)+A_{R_n}(b,d)\Bigr]\notag\\
 &=L_{R_n}(a,c)-L_{R_n}(a,d)
 -L_{R_n}(b,c)+L_{R_n}(b,d)+O(1).
 \label{eq:image-log-decomposition}
\end{align}
Since $a,c\in(-1,1)$,
\begin{equation*}
 \operatorname{dist}((a,c),2\Z^2)=\sqrt{a^2+c^2}=|x-y|.
\end{equation*}
Moreover, for $x_1,y_1\in(0,1)$,
\begin{equation*}
 |a|\le \min\{b,2-b\}
 =\operatorname{dist}(b,2\Z),
\end{equation*}
and similarly $|c|\le\operatorname{dist}(d,2\Z)$.  Hence
\begin{equation}
 \operatorname{dist}((a,d),2\Z^2)
 \le \operatorname{dist}((b,d),2\Z^2).
 \label{eq:reflected-log-order}
\end{equation}
Since $r\mapsto\log^+(1/(r\vee R_n^{-1}))$ is non-increasing,
\eqref{eq:reflected-log-order} gives
$L_{R_n}(a,d)\ge L_{R_n}(b,d)$.  Because $L_{R_n}(b,c)\ge0$,
\[
 -L_{R_n}(a,d)-L_{R_n}(b,c)+L_{R_n}(b,d)
 =-\bigl(L_{R_n}(a,d)-L_{R_n}(b,d)\bigr)-L_{R_n}(b,c)\le0.
\]
Thus all reflected contributions can be discarded in the upper bound, and
\[
 \sum_{\lambda_k\le2^{2n}}\psi_k(x)\psi_k(y)
 \le L_{R_n}(a,c)+C
 \le \log^+\frac1{|x-y|}+C,
\]
which is \eqref{eq:spectral-covariance-bound}.  If $x,y$ range in a fixed
compact $K\Subset D$, then the three reflected image points in
\eqref{eq:image-log-decomposition} stay a positive distance from
$2\Z^2$.  Their logarithmic terms are therefore uniformly bounded.  Since
$R_n^{-1}=\pi2^{-n}$, changing the cutoff scale from $R_n^{-1}$ to $2^{-n}$
changes the logarithm by at most a constant.  This proves
\eqref{eq:spectral-covariance-local}.
\end{proof}

Let $\mathcal M^+_{\rm fin}(D)$ be the space of finite positive Radon
measures on $D$, equipped with the narrow topology, and fix a complete
compatible metric $d_{\rm w}$.  Since each $M_{n,b}$ depends on finitely many
coefficients, the map $b\mapsto M_{n,b}$ is measurable.  Set
\begin{equation*}
 \Omega_{\rm can}
 :=\{b\in\Omega:(M_{n,b})_{n\ge1}\text{ is }d_{\rm w}\text{-Cauchy}\}.
\end{equation*}
Thus $\Omega_{\rm can}\in\Borel(\Omega)$ is precisely the canonical
convergence event for the approximating measures.  Define on all of $\Omega$
\begin{equation*}
 M_b:=
 \begin{cases}
  \displaystyle\lim_{n\to\infty}M_{n,b},&b\in\Omega_{\rm can},\\[1mm]
  0,&b\notin\Omega_{\rm can},
 \end{cases}
\end{equation*}
where the limit is taken in $d_{\rm w}$.  This fixes a canonical speed-measure
realization for every coefficient environment; the next proposition records
its basic probabilistic properties.

\begin{proposition}
\label{prop:chaos}
As defined above, $\Omega_{\rm can}$ and $M_b$ satisfy
\begin{equation*}
 \Prob(\Omega_{\rm can})=1.
\end{equation*}
The map
\begin{equation*}
 b\longmapsto M_b:
 (\Omega,\Borel(\Omega))
 \longrightarrow
 \bigl(\mathcal M^+_{\rm fin}(D),\Borel(\mathcal M^+_{\rm fin}(D))\bigr)
\end{equation*}
is measurable, and
\begin{equation*}
 M_{n,b}\Longrightarrow M_b,\qquad b\in\Omega_{\rm can}.
\end{equation*}
Moreover, for every $g\in C(\overline D)$,
\begin{equation}
 \int_Dg\,dM_{n,b}\longrightarrow\int_Dg\,dM_b
 \quad\text{almost surely and in }L^2(\Prob),
 \label{eq:test-L2}
\end{equation}
and
\begin{equation}
 \Ex\int_Dg\,dM_b=\int_Dg(x)\,dx.
 \label{eq:first-moment}
\end{equation}
\end{proposition}

\begin{proof}
Set
\begin{equation*}
 W_{n,b}(x):=\frac{e^{\gamma S_{n,b}(x)}}{Z_{n,\gamma}(x)}.
\end{equation*}
Let $\mathcal F_n$ be the sigma-field generated by the coefficients appearing
in $S_{n,b}$.  Because the newly added coefficients are independent and the
normalization in \eqref{eq:chaos-approx} factorizes coordinate by coordinate,
\begin{equation*}
 \Ex[W_{n+1,b}(x)\mid\mathcal F_n]=W_{n,b}(x)
 \qquad(x\in D).
\end{equation*}
Hence, for every bounded Borel $g$,
\begin{equation*}
 X_n(g):=\int_Dg(x)W_{n,b}(x)\,dx
\end{equation*}
is an $\mathcal F_n$-martingale.

We now establish a uniform second-moment bound.  Introduce the logarithmic
moment generating function of one coefficient,
\begin{equation*}
 \chi(t):=\log\Ex e^{t b_1},\qquad t\in\R.
\end{equation*}
By \eqref{eq:coefficient-law},
\[
 \chi(0)=\chi'(0)=0,\qquad \chi''(0)=1.
\]
Thus the quadratic part of $\chi$ produces the covariance term, while the
higher-order terms contribute a summable remainder.  Indeed, since
$\|\psi_k\|_\infty\le C\lambda_k^{-1/2}$ and
$\sum_k\lambda_k^{-3/2}<\infty$, Taylor's theorem gives, for all sufficiently
large $k$,
\begin{equation*}
 \chi(u+v)-\chi(u)-\chi(v)
 =uv+O\!\left(|uv|(|u|+|v|)\right),
\end{equation*}
uniformly for
\[
 u=\gamma\psi_k(x),\qquad v=\gamma\psi_k(y),
 \qquad x,y\in D.
\]
The error terms are summable in $k$, while the finitely many remaining
coordinates contribute a bounded constant.  Independence therefore yields
\begin{align*}
 \log\Ex[W_{n,b}(x)W_{n,b}(y)]
 &=\sum_{\lambda_k\le2^{2n}}
 \Bigl[
 \chi\bigl(\gamma(\psi_k(x)+\psi_k(y))\bigr)
 -\chi(\gamma\psi_k(x))-\chi(\gamma\psi_k(y))
 \Bigr]\\
 &\le C_\gamma+\gamma^2K_n(x,y).
\end{align*}
By \cref{lem:spectral-covariance},
\begin{equation}
 \Ex[W_{n,b}(x)W_{n,b}(y)]
 \le C_\gamma\left(1+|x-y|^{-\gamma^2}\right).
 \label{eq:two-point}
\end{equation}
Since $\gamma^2<2$,
\begin{equation*}
 I_\gamma:=
 \iint_{D^2}\left(1+|x-y|^{-\gamma^2}\right)\,dx\,dy<\infty.
\end{equation*}
Consequently,
\begin{equation*}
 \sup_n\Ex|X_n(g)|^2
 \le C_\gamma I_\gamma\|g\|_\infty^2.
\end{equation*}

Choose a countable uniformly dense vector lattice
$\mathscr C_0\subset C(\overline D)$ containing $1$.  By martingale
convergence, after intersecting countably many probability-one events,
$X_n(h)$ converges almost surely and in $L^2$ for every
$h\in\mathscr C_0$.  On the same event $X_n(1)$ converges, and hence
\begin{equation*}
 C_{\rm mass}(b):=\sup_{n\ge1}X_n(1)<\infty.
\end{equation*}
For $g\in C(\overline D)$ and $h\in\mathscr C_0$,
\begin{equation}
 |X_n(g)-X_n(h)|
 \le \|g-h\|_\infty X_n(1)
 \le C_{\rm mass}(b)\|g-h\|_\infty.
 \label{eq:pathwise-test-continuity}
\end{equation}
Choose $h_j\in\mathscr C_0$ with $\|h_j-g\|_\infty\to0$.  Since
$X_n(h_j)$ converges for every $j$, \eqref{eq:pathwise-test-continuity}
shows that $(X_n(g))_n$ is Cauchy and therefore converges almost surely.
Denote its limit by $L_b(g)$.  Linearity passes to the limit, and the same
estimate gives
\begin{equation*}
 |L_b(g)|\le C_{\rm mass}(b)\|g\|_\infty,
 \qquad
 L_b(g)\ge0\quad(g\ge0).
\end{equation*}
Thus the Riesz representation theorem gives a finite random measure
$\overline M_b$ on $\overline D$ such that
$L_b(g)=\int_{\overline D}g\,d\overline M_b$.

The second-moment estimate also yields, for all
$g,h\in C(\overline D)$,
\begin{equation*}
 \sup_n\Ex|X_n(g)-X_n(h)|^2
 \le C_\gamma I_\gamma\|g-h\|_\infty^2.
\end{equation*}
By Fatou's lemma the same bound holds for
$L_b(g)-L_b(h)$.  Hence, for $h_j\in\mathscr C_0$ with
$\|h_j-g\|_\infty\to0$,
\begin{equation*}
 \limsup_{n\to\infty}
 \|X_n(g)-L_b(g)\|_{L^2(\Prob)}
 \le 2(C_\gamma I_\gamma)^{1/2}
 \|g-h_j\|_\infty.
\end{equation*}
Letting $j\to\infty$ gives
\begin{equation}
 X_n(g)\longrightarrow\int_{\overline D}g\,d\overline M_b
 \quad\text{almost surely and in }L^2(\Prob),\qquad g\in C(\overline D).
 \label{eq:closed-square-limit}
\end{equation}
Because $\Ex W_{n,b}(x)=1$, passage to expectations gives
\begin{equation}
 \Ex\int_{\overline D}g\,d\overline M_b=\int_Dg(x)\,dx
 \qquad(g\in C(\overline D)).
 \label{eq:closed-square-first-moment}
\end{equation}

To exclude boundary mass, choose a decreasing sequence
$\varepsilon_j\downarrow0$ and functions
$\rho_j\in C(\overline D)$ satisfying
\[
 0\le\rho_j\le1,\qquad
 \rho_j=1\ \text{on }\partial D,\qquad
 \supp\rho_j\subset
 \{x:\dist(x,\partial D)<\varepsilon_j\},
 \qquad \rho_j\downarrow\mathbf1_{\partial D}.
\]
Then \eqref{eq:closed-square-first-moment} implies
\[
 \Ex\int_{\overline D}\rho_j\,d\overline M_b
 \le |\{x\in D:\dist(x,\partial D)<\varepsilon_j\}|.
\]
Letting $j\to\infty$, using continuity from above pathwise and Fatou's
lemma, gives
\begin{equation}
 \overline M_b(\partial D)=0\qquad\text{almost surely}.
 \label{eq:no-boundary-mass}
\end{equation}
For the closed strips
\[
 F_\varepsilon:=\{x\in\overline D:\dist(x,\partial D)\le\varepsilon\},
\]
Portmanteau and continuity from above give
\begin{equation}
 \lim_{\varepsilon\downarrow0}\limsup_{n\to\infty}
 M_{n,b}(F_\varepsilon\cap D)=0
 \qquad\text{almost surely}.
 \label{eq:boundary-tightness}
\end{equation}
If $g\in C_b(D)$, multiply $g$ by a continuous cutoff
$\eta_\varepsilon$ which vanishes in the $\varepsilon$-boundary strip and is
one outside the $2\varepsilon$-strip.  The product
$g\eta_\varepsilon$, extended by zero to $\partial D$, belongs to
$C(\overline D)$.  Using \eqref{eq:closed-square-limit} first and then
\eqref{eq:no-boundary-mass}--\eqref{eq:boundary-tightness} shows
\[
 \int_Dg\,dM_{n,b}\longrightarrow
 \int_Dg\,d(\overline M_b|_D).
\]
This proves narrow convergence on $D$.  In particular,
$\Prob(\Omega_{\rm can})=1$, and uniqueness of the narrow limit identifies
$\overline M_b|_D$ with the $M_b$ defined above.  The pointwise limit on $\Omega_{\rm can}$, extended by zero on its
complement, is measurable.  Equation \eqref{eq:test-L2} follows from
\eqref{eq:closed-square-limit} and the boundary argument above, and
\eqref{eq:first-moment} follows by taking expectations.
\end{proof}

The same two-point estimate yields uniform small-ball control.  For $m\ge0$
let $\mathcal D_m$ be the finite family of dyadic squares of side $2^{-m}$
that meet $D$.  For
\begin{equation}
 0<\alpha<1-\frac{\gamma^2}{2},
 \label{eq:alpha-range}
\end{equation}
set
\begin{equation*}
 C_{{\rm Fr},\alpha}(b)
 :=\sup_{m\ge0}2^{m\alpha}
   \max_{R\in\mathcal D_m}M_b(R\cap D),
 \qquad
 \Omega_{\rm Fr}(\alpha)
 :=\Omega_{\rm can}\cap\{C_{{\rm Fr},\alpha}<\infty\}.
\end{equation*}
This quantity records the uniform dyadic mass bound from which the Frostman
estimate follows.  Since $b\mapsto M_b$ is measurable and evaluation of a
finite Radon measure on a fixed Borel set is measurable,
$C_{{\rm Fr},\alpha}$ is a measurable countable supremum of finite maxima,
and hence $\Omega_{\rm Fr}(\alpha)\in\Borel(\Omega)$.

\begin{lemma}
\label{lem:frostman}
As defined above, $\Omega_{\rm Fr}(\alpha)$ satisfies
\begin{equation*}
 \Prob(\Omega_{\rm Fr}(\alpha))=1.
\end{equation*}
Moreover, for every $b\in\Omega_{\rm Fr}(\alpha)$,
\begin{equation}
 M_b(B(x,r)\cap D)
 \le C_\alpha C_{{\rm Fr},\alpha}(b)r^\alpha,
 \qquad x\in D,\quad0<r\le1,
 \label{eq:frostman}
\end{equation}
where $C_\alpha<\infty$ is deterministic.
\end{lemma}

\begin{proof}
Fix $m\ge0$ and $R\in\mathcal D_m$, and choose
$\eta_R\in C(\overline D)$ such that
\[
 0\le\eta_R\le1,\qquad
 \eta_R=1\ \text{on }R\cap D,
\]
with support contained in the concentric square of side $3\ell(R)$.
By \eqref{eq:test-L2} and \eqref{eq:two-point},
\begin{align*}
 \Ex[M_b(R\cap D)^2]
 &\le \Ex\left(\int_D\eta_R\,dM_b\right)^2\\
 &\le C_\gamma
 \iint_{\supp\eta_R\times\supp\eta_R}
 \left(1+|x-y|^{-\gamma^2}\right)\,dx\,dy\\
 &\le C_\gamma \ell(R)^{4-\gamma^2}.
\end{align*}
Since $\#\mathcal D_m\le C2^{2m}$, Markov's inequality gives
\begin{align*}
 &\Prob\left\{
 \max_{R\in\mathcal D_m}M_b(R\cap D)>2^{-m\alpha}\right\}\\
 &\qquad\le
 C2^{2m}\,2^{2m\alpha}\,2^{-m(4-\gamma^2)}
 =C2^{-m(2-\gamma^2-2\alpha)}.
\end{align*}
The exponent is positive by \eqref{eq:alpha-range}, so Borel--Cantelli yields,
for almost every $b$, an $m_0(b)$ such that
\begin{equation}
 M_b(R\cap D)\le2^{-m\alpha},
 \qquad m\ge m_0(b),\quad R\in\mathcal D_m.
 \label{eq:dyadic-frostman}
\end{equation}
Because $M_b(D)<\infty$, the finitely many levels $m<m_0(b)$ contribute a
finite amount to $C_{{\rm Fr},\alpha}(b)$.  Hence
$C_{{\rm Fr},\alpha}(b)<\infty$ almost surely, and therefore
\[
 \Prob(\Omega_{\rm Fr}(\alpha))=1.
\]
If $2^{-m-1}<r\le2^{-m}$, the ball $B(x,r)$ meets at most a deterministic
number $N_0$ of squares in $\mathcal D_m$.  Consequently,
\begin{align*}
 M_b(B(x,r)\cap D)
 &\le N_0\max_{R\in\mathcal D_m}M_b(R\cap D)\\
 &\le N_0 C_{{\rm Fr},\alpha}(b)2^{-m\alpha}
 \le N_0 2^\alpha C_{{\rm Fr},\alpha}(b)r^\alpha,
\end{align*}
which proves \eqref{eq:frostman} with $C_\alpha=N_0 2^\alpha$.
\end{proof}

\subsection{Finite coefficient orbits and full support}
\label{subsec:finite-coefficient-orbits}

Define
\begin{equation}
 \mathcal T:=\Span_\R\{e_k:k\ge1\},
 \qquad
 Q:=\Span_\Q\{e_k:k\ge1\}.
 \label{eq:finite-mode-cores}
\end{equation}
If $f=\sum_{k\in I}c_ke_k\in\mathcal T$, where $I$ is finite, set
\begin{equation*}
 \delta(f)_k:=
 \begin{cases}
 (2\pi)^{-1/2}\lambda_k^{1/2}c_k,&k\in I,\\
 0,&k\notin I.
 \end{cases}
\end{equation*}
The map $\delta$ identifies a finite coefficient translation with an
additive shift of the truncated field.

\begin{corollary}
\label{cor:canonical-orbit}
If $f\in\mathcal T$, $t\in\R$, and
$b,b+t\delta(f)\in\Omega_{\rm can}$, then
\begin{equation}
 M_{b+t\delta(f)}=e^{\gamma tf}M_b.
 \label{eq:canonical-orbit}
\end{equation}
\end{corollary}

\begin{proof}
For all sufficiently large $n$, every mode occurring in $f$ is present in
\eqref{eq:field}, and hence
\begin{equation*}
 S_{n,b+t\delta(f)}=S_{n,b}+tf,
 \qquad
 M_{n,b+t\delta(f)}=e^{\gamma tf}M_{n,b}.
\end{equation*}
Thus, for $g\in C_b(D)$,
\begin{align*}
 \int_D g\,dM_{b+t\delta(f)}
 &=\lim_{n\to\infty}\int_D g\,dM_{n,b+t\delta(f)}\\
 &=\lim_{n\to\infty}\int_D ge^{\gamma tf}\,dM_{n,b}
 =\int_D ge^{\gamma tf}\,dM_b,
\end{align*}
which proves \eqref{eq:canonical-orbit}.
\end{proof}

The orbit identity also prevents the limiting measure from vanishing on an
open set.

\begin{lemma}
\label{lem:full-support}
Almost surely,
\begin{equation}
 M_b(O)>0
 \qquad\text{for every non-empty open }O\subset D.
 \label{eq:full-support}
\end{equation}
Equivalently, $\supp M_b=D$.
\end{lemma}

\begin{proof}
Fix $0\le g\in C_c(D)$, $g\not\equiv0$, and set
\[
 X_g(b):=\int_D g\,dM_b,
 \qquad
 E_g:=\{X_g=0\}.
\]
By \cref{prop:chaos}, $X_g$ is measurable, and \eqref{eq:first-moment} gives
\begin{equation}
 \Ex X_g=\int_Dg(x)\,dx>0.
 \label{eq:positive-test-mean}
\end{equation}

We show that $E_g$ belongs to the completion of the coefficient tail
$\sigma$-field.  Fix $d\ge1$ and write
$b=(u,y)\in\R^d\times\R^{\N\setminus\{1,\ldots,d\}}$.  Since
$\Prob(\Omega_{\rm can})=1$, Fubini gives, for almost every $y$,
\[
 C_y:=\{u:(u,y)\in\Omega_{\rm can}\},
 \qquad
 \mathsf p^{\otimes d}(C_y)=1.
\]
For $u,v\in C_y$, put
\begin{equation*}
 f_{u-v}:=\sqrt{2\pi}\sum_{k=1}^d
 (u_k-v_k)\lambda_k^{-1/2}e_k\in\mathcal T.
\end{equation*}
Then $\delta(f_{u-v})=u-v$, and \cref{cor:canonical-orbit} gives
\[
 M_{(u,y)}=e^{\gamma f_{u-v}}M_{(v,y)}.
\]
Since $g\ge0$ and $e^{\gamma f_{u-v}}>0$,
\begin{equation*}
 X_g(u,y)=0\quad\Longleftrightarrow\quad X_g(v,y)=0.
\end{equation*}
Thus $u\mapsto\mathbf1_{E_g}(u,y)$ is
$\mathsf p^{\otimes d}$-almost surely constant for almost every $y$.  Hence
$\mathbf1_{E_g}$ admits an $\mathcal F_d^{\rm tail}$-measurable version, where
\[
 \mathcal F_d^{\rm tail}:=\sigma(b_k:k>d),
 \qquad
 \mathcal F_\infty^{\rm tail}:=\bigcap_{d\ge1}\mathcal F_d^{\rm tail}.
\]
Therefore
\[
 \Ex[\mathbf1_{E_g}\mid\mathcal F_d^{\rm tail}]=\mathbf1_{E_g}
 \qquad\text{almost surely for every }d.
\]
The reverse martingale theorem yields
\[
 \mathbf1_{E_g}=\Ex[\mathbf1_{E_g}\mid\mathcal F_\infty^{\rm tail}]
 \qquad\text{almost surely},
\]
so $E_g$ is measurable with respect to the completed tail $\sigma$-field.
Kolmogorov's zero--one law therefore gives
\[
 \Prob(E_g)\in\{0,1\}.
\]
Equation \eqref{eq:positive-test-mean} rules out $\Prob(E_g)=1$, and hence
\[
 \Prob(X_g>0)=1.
\]

Choose a countable family $(g_j)_{j\ge1}\subset C_c(D)$ of non-negative,
non-zero functions such that every non-empty open subset of $D$ contains the
support of some $g_j$.  On
\[
 \bigcap_{j\ge1}\{X_{g_j}>0\},
\]
which has probability one, every non-empty open set has positive $M_b$-mass.
This proves \eqref{eq:full-support}.
\end{proof}

\subsection{Trace form and spectral realization}

This subsection fixes the measure class used in the pathwise spectral argument.
We first obtain a full-probability structural core by comparison with Gaussian
chaos, then close it under finite-mode tilts and realize the resulting trace
operators through the Dirichlet Green kernel.

Let
\begin{equation*}
 L_0:=-(2\pi)^{-1}\Delta_D,
 \qquad
 G_D(x,y):=\int_0^\infty q_t^D(x,y)\,dt,
 \qquad
 -\Delta_xG_D(x,y)=2\pi\delta_y,
\end{equation*}
where $q_t^D$ is the killed heat kernel of $L_0$.  The properties not obtained
directly from the product coefficients will be transferred from a
covariance-matched Gaussian chaos.  Let
\begin{equation*}
 \Omega^{\rm G}:=\R^{\N},\qquad
 g=(g_k)_{k\ge1},\qquad
 \Prob^{\rm G}:=\mathcal N(0,1)^{\otimes\N},
\end{equation*}
and set
\begin{equation*}
 S_{n,g}^{\rm G}(x):=\sum_{\lambda_k\le2^{2n}}g_k\psi_k(x),
 \qquad
 M_{n,g}^{\rm G}(dx):=
 \frac{e^{\gamma S_{n,g}^{\rm G}(x)}}
 {\Ex^{\rm G} e^{\gamma S_{n,g}^{\rm G}(x)}}\,dx.
\end{equation*}
Using the metric $d_{\rm w}$ fixed above, define
\begin{equation*}
 \Omega_{\rm can}^{\rm G}
 :=\{g\in\Omega^{\rm G}:(M_{n,g}^{\rm G})_{n\ge1}
 \text{ is }d_{\rm w}\text{-Cauchy}\}
\end{equation*}
and, on all of $\Omega^{\rm G}$,
\begin{equation*}
 M_g^{\rm G}:=
 \begin{cases}
  \displaystyle\lim_{n\to\infty}M_{n,g}^{\rm G},
  &g\in\Omega_{\rm can}^{\rm G},\\[1mm]
  0,&g\notin\Omega_{\rm can}^{\rm G}.
 \end{cases}
\end{equation*}
This fixes the Gaussian reference on its own coefficient space.

\begin{lemma}
\label{lem:gaussian-reference}
As defined above, $\Omega_{\rm can}^{\rm G}$ and $M_g^{\rm G}$ satisfy
\begin{equation*}
 \Prob^{\rm G}(\Omega_{\rm can}^{\rm G})=1.
\end{equation*}
The map
\begin{equation*}
 g\longmapsto M_g^{\rm G}:
 (\Omega^{\rm G},\Borel(\Omega^{\rm G}))
 \longrightarrow
 \bigl(\mathcal M^+_{\rm fin}(D),\Borel(\mathcal M^+_{\rm fin}(D))\bigr)
\end{equation*}
is measurable, and
\begin{equation*}
 M_{n,g}^{\rm G}\Longrightarrow M_g^{\rm G},
 \qquad g\in\Omega_{\rm can}^{\rm G}.
\end{equation*}
Moreover, for every $h\in C(\overline D)$,
\begin{equation*}
 \int_D h\,dM_{n,g}^{\rm G}
 \longrightarrow
 \int_D h\,dM_g^{\rm G}
 \qquad\text{$\Prob^{\rm G}$-almost surely and in }L^2(\Prob^{\rm G}).
\end{equation*}
\end{lemma}

\begin{proof}
Put
\[
 W_{n,g}^{\rm G}(x):=
 \frac{e^{\gamma S_{n,g}^{\rm G}(x)}}
 {\Ex^{\rm G} e^{\gamma S_{n,g}^{\rm G}(x)}}.
\]
With respect to the filtration generated by the active Gaussian modes,
$\int_D hW_{n,g}^{\rm G}\,dx$ is a martingale, and Gaussianity gives
\begin{equation*}
 \Ex^{\rm G}\!\left[W_{n,g}^{\rm G}(x)W_{n,g}^{\rm G}(y)\right]
 =\exp\!\bigl(\gamma^2K_n(x,y)\bigr).
\end{equation*}
Hence \cref{lem:spectral-covariance} yields
\begin{equation*}
 \Ex^{\rm G}\!\left[W_{n,g}^{\rm G}(x)W_{n,g}^{\rm G}(y)\right]
 \le C_\gamma\bigl(1+|x-y|^{-\gamma^2}\bigr).
\end{equation*}
Since $\gamma^2<2$,
\begin{equation*}
 \sup_n\Ex^{\rm G}\left|\int_D h\,dM_{n,g}^{\rm G}\right|^2
 \le C_\gamma\|h\|_\infty^2
 \iint_{D^2}\bigl(1+|x-y|^{-\gamma^2}\bigr)\,dx\,dy<\infty.
\end{equation*}
The dense-test-function and boundary-tightness argument in the proof of
\cref{prop:chaos} applies verbatim.  It gives the asserted convergence,
$\Prob^{\rm G}(\Omega_{\rm can}^{\rm G})=1$, and measurability of the
canonical limit map.
\end{proof}

The covariance kernels of the Gaussian approximants satisfy
\begin{equation*}
 K_n(x,y)
 =2\pi\sum_{\lambda_k\le2^{2n}}\lambda_k^{-1}e_k(x)e_k(y),
 \qquad
 G_D(x,y)=2\pi\sum_{k\ge1}\lambda_k^{-1}e_k(x)e_k(y),
\end{equation*}
and therefore
\begin{equation*}
 \|K_n-G_D\|_{L^2(D^2)}^2
 =(2\pi)^2\sum_{\lambda_k>2^{2n}}\lambda_k^{-2}
 \longrightarrow0.
\end{equation*}
The full Gaussian spectral series is the Karhunen--Lo\`eve realization of
the zero-boundary Dirichlet GFF with covariance kernel $G_D$.  The uniqueness
and approximation independence of subcritical Gaussian multiplicative chaos,
as recalled in \cite[Section~6.1]{CaiResponseSubmitted}, therefore identify
$M_g^{\rm G}$ with the Dirichlet GMC used there.  In particular, the
singularity and quasi-support results quoted below apply to $M_g^{\rm G}$.

\begin{lemma}
\label{lem:gaussian-comparison}
There exists a coupling $\mathbf P$ of $\Prob$ and $\Prob^{\rm G}$ on
$\Omega\times\Omega^{\rm G}$ such that
\begin{equation}
 M_b\sim M_g^{\rm G}
 \quad\text{for $\mathbf P$-almost every }(b,g).
 \label{eq:gaussian-comparison}
\end{equation}
\end{lemma}

\begin{proof}
For every $\lambda,t>0$,
\begin{equation*}
 \Prob(|b_1|>t)
 \le e^{-\lambda t}
 \bigl(\Ex e^{\lambda b_1}+\Ex e^{-\lambda b_1}\bigr),
\end{equation*}
so \eqref{eq:coefficient-law} gives the stretched-exponential tail condition
of \cite{BasuGanguly}.  In dimension $d=2$, the unit-cube argument in
\cite[Section~3.4]{BasuGanguly} uses
\begin{equation*}
 \lambda_k^{-d/4}=\lambda_k^{-1/2},
 \qquad
 \lambda_k\in\bigl(2^{2(n-1)},2^{2n}\bigr],
\end{equation*}
which are respectively the weights and cutoff increments used here.  Moreover,
for every compact $K\Subset D$, \cref{lem:spectral-covariance} gives
\begin{equation*}
 K_n(x,y)
 =\log^+\!\frac1{|x-y|\vee2^{-n}}+O_K(1),
 \qquad x,y\in K,
\end{equation*}
so the logarithmic coefficient is one.  Since $0<\gamma<\sqrt2<2$,
\cite[Section~3.4]{BasuGanguly} yields a coupling for which the two limiting
chaos measures are equivalent.  The marginal convergence events in
\cref{prop:chaos,lem:gaussian-reference} remain of full measure under this
coupling, so these limits are the canonical versions $M_b$ and $M_g^{\rm G}$
fixed above.  This proves \eqref{eq:gaussian-comparison}.
\end{proof}

Fix an exponent $\alpha$ satisfying \eqref{eq:alpha-range}, and let
$(O_j)_{j\ge1}$ be a countable base of non-empty open subsets of $D$.

\begin{proposition}
\label{prop:structural}
There exists $\Omega_*\in\Borel(\Omega)$ such that
\begin{equation*}
 \Omega_*\subset\Omega_{\rm can}\cap\Omega_{\rm Fr}(\alpha),
 \qquad
 \Prob(\Omega_*)=1,
\end{equation*}
and the following assertions hold for every $b\in\Omega_*$.
\begin{enumerate}[label=\textnormal{(\roman*)}]
\item
\begin{equation}
 0<M_b(D)<\infty,\qquad M_b\perp dx,\qquad \supp M_b=D,\qquad
 M_b(\{x\})=0\quad(x\in D),
 \label{eq:structural-measure}
\end{equation}
and
\begin{equation*}
 M_b(B(x,r)\cap D)
 \le C_\alpha C_{{\rm Fr},\alpha}(b)r^\alpha,
 \qquad x\in D,\quad0<r\le1.
\end{equation*}
\item For every $\E$-polar set $N$ and every $u\in H_0^1(D)$,
\begin{equation}
 M_b(N)=0,\qquad
 \widetilde u=0\ M_b\text{-a.e.}
 \Longrightarrow u=0\quad\E\text{-q.e.}
 \label{eq:structural-quasi-support}
\end{equation}
Thus $M_b$ is smooth and has full quasi-support.
\item The time-changed form
\begin{equation}
 V_b=\{u\in H_0^1(D):\widetilde u\in L^2(M_b)\},
 \qquad (\E,V_b)
 \label{eq:trace-form}
\end{equation}
is densely defined and closed on $L^2(M_b)$.
\end{enumerate}
\end{proposition}

\begin{proof}
Set
\begin{equation*}
 \Omega_{\rm dir}:=
 \Omega_{\rm can}\cap\Omega_{\rm Fr}(\alpha)
 \cap\bigcap_{j\ge1}\{M_b(O_j)>0\}.
\end{equation*}
Since $b\mapsto M_b$ is measurable and evaluation on a fixed open set is
measurable,
\begin{equation*}
 \Omega_{\rm dir}\in\Borel(\Omega),
 \qquad
 \Prob(\Omega_{\rm dir})=1
\end{equation*}
by \cref{prop:chaos,lem:frostman,lem:full-support}.

The identification above and
\cite[Lemmas~6.6 and 6.28]{CaiResponseSubmitted} give a
$\Prob^{\rm G}$-completed full event on which
\begin{equation*}
 M_g^{\rm G}\perp dx,
 \qquad
 M_g^{\rm G}(N)=0\quad\text{for every $\E$-polar set }N,
\end{equation*}
and
\begin{equation*}
 \widetilde u=0\ M_g^{\rm G}\text{-a.e.}
 \Longrightarrow u=0\quad\E\text{-q.e.},
 \qquad u\in H_0^1(D).
\end{equation*}
Choose a Borel subset $\Omega_{\rm tr}^{\rm G}\subset\Omega_{\rm can}^{\rm G}$
of full $\Prob^{\rm G}$-measure on which these conclusions hold.

Let $\mathbf P$ be the coupling from \cref{lem:gaussian-comparison}.  Under
$\mathbf P$, the conditions
\begin{equation*}
 b\in\Omega_{\rm dir},\qquad
 g\in\Omega_{\rm tr}^{\rm G},\qquad
 M_b\sim M_g^{\rm G}
\end{equation*}
hold on a completed full event.  Choose
\begin{equation*}
 E\in\Borel(\Omega\times\Omega^{\rm G}),
 \qquad
 \mathbf P(E)=1,
\end{equation*}
inside that event, and put
\begin{equation*}
 \Omega_{\rm an}:=\operatorname{proj}_{\Omega}E.
\end{equation*}
The set $\Omega_{\rm an}$ is analytic and therefore universally measurable.  Since the first
marginal of $\mathbf P$ is $\Prob$,
\begin{equation*}
 1=\mathbf P(E)
 \le\mathbf P(\Omega_{\rm an}\times\Omega^{\rm G})
 =\Prob(\Omega_{\rm an}),
\end{equation*}
so $\Prob(\Omega_{\rm an})=1$.  Moreover,
$\Omega_{\rm an}\subset\Omega_{\rm dir}$.  As $\Omega$ is a
standard Borel space, there is a Borel set
\begin{equation*}
 \Omega_*\subset\Omega_{\rm an},
 \qquad
 \Prob(\Omega_*)=1.
\end{equation*}

Fix $b\in\Omega_*$.  Since $\Omega_*\subset\Omega_{\rm an}\subset\Omega_{\rm dir}$,
\cref{prop:chaos,lem:frostman,lem:full-support} give
\begin{equation*}
 0<M_b(D)<\infty,
 \qquad
 \supp M_b=D,
\end{equation*}
together with the Frostman bound in \textnormal{(i)}; letting $r\downarrow0$
in that bound gives $M_b(\{x\})=0$ for every $x\in D$.

By the definition of $\Omega_{\rm an}$, there exists $g\in\Omega_{\rm tr}^{\rm G}$ with
$(b,g)\in E$.  Hence $M_b\sim M_g^{\rm G}$, and therefore
\begin{equation*}
 M_b\perp dx,
 \qquad
 M_b(N)=0\quad\text{for every $\E$-polar set }N,
\end{equation*}
while
\begin{equation*}
 \widetilde u=0\ M_b\text{-a.e.}
 \Longrightarrow
 \widetilde u=0\ M_g^{\rm G}\text{-a.e.}
 \Longrightarrow u=0\quad\E\text{-q.e.}
\end{equation*}
for every $u\in H_0^1(D)$.  This proves \textnormal{(i)}--\textnormal{(ii)}.
The full-quasi-support time-change theorem
\cite[Theorem~6.2.1 and (6.2.22)]{FukushimaOshimaTakeda} gives
\textnormal{(iii)}.
\end{proof}

Fix one such $\Omega_*$ for the remainder of the article and define the
finite-mode orbit class
\begin{equation}
 \mathcal M_*:=\{e^{\gamma f}M_b:\ b\in\Omega_*,\ f\in\mathcal T\}.
 \label{eq:orbit-class}
\end{equation}
These are precisely the measures used in the deterministic response argument.

\begin{corollary}
\label{cor:structural-orbit}
Every $\mu\in\mathcal M_*$ satisfies the conclusions of
\cref{prop:structural}\textnormal{(i)}--\textnormal{(iii)}.  In the Frostman
bound one may write
\begin{equation*}
 \mu(B(x,r)\cap D)\le C_\mu r^\alpha,
 \qquad x\in D,\quad0<r\le1,
\end{equation*}
for some $C_\mu<\infty$.
\end{corollary}

\begin{proof}
Write $\mu=e^{\gamma f}M_b$ with $b\in\Omega_*$ and $f\in\mathcal T$.  Since
$f$ is bounded,
\begin{equation*}
 e^{-\gamma\|f\|_\infty}M_b
 \le \mu
 \le e^{\gamma\|f\|_\infty}M_b.
\end{equation*}
Thus $\mu$ and $M_b$ have the same null sets and support, while
\[
 C_\mu:=e^{\gamma\|f\|_\infty}
 C_\alpha C_{{\rm Fr},\alpha}(b)<\infty
\]
gives the stated Frostman bound.  The remaining assertions in
\textnormal{(i)}--\textnormal{(ii)} follow from \cref{prop:structural}, and the
full-quasi-support time-change theorem gives \textnormal{(iii)}.
\end{proof}

For $\mu\in\mathcal M_*$ write
\begin{equation*}
 V_\mu:=\{u\in H_0^1(D):\widetilde u\in L^2(\mu)\},
 \qquad
 \|u\|_{\E,\mu}^2:=\E(u,u)+\|u\|_{L^2(\mu)}^2,
\end{equation*}
and let $A_\mu$ be the non-negative self-adjoint operator associated with the
closed form $(\E,V_\mu)$ on $L^2(\mu)$.

For $x,y\in D$ put
\begin{equation*}
 \ell_x(y):=1+\log^+\frac1{|x-y|}.
\end{equation*}
The Frostman estimate on $\mathcal M_*$ gives the logarithmic control needed
for the Green kernel.

\begin{lemma}
\label{lem:log-square}
For every $\mu\in\mathcal M_*$,
\begin{equation}
 \sup_{x\in D}\int_D\ell_x(y)^2\,d\mu(y)<\infty,
 \label{eq:log-square-bound}
\end{equation}
and
\begin{equation}
 \lim_{\rho\downarrow0}\sup_{x\in D}
 \int_{B(x,\rho)\cap D}\ell_x(y)^2\,d\mu(y)=0.
 \label{eq:log-square-local}
\end{equation}
\end{lemma}

\begin{proof}
By \cref{cor:structural-orbit}, for some $C_\mu<\infty$,
\begin{equation*}
 \mu(B(x,r)\cap D)\le C_\mu r^\alpha,
 \qquad x\in D,\quad0<r\le1.
\end{equation*}
For $k\ge0$ set
\begin{equation*}
 \mathcal R_k(x):=(B(x,2^{-k})\cap D)\setminus B(x,2^{-k-1}).
\end{equation*}
Since $\ell_x\le C(k+1)$ on $\mathcal R_k(x)$,
\begin{align*}
 \int_D\ell_x(y)^2\,d\mu(y)
 &\le C\mu(D)+C\sum_{k\ge0}(k+1)^2\mu(B(x,2^{-k})\cap D)\\
 &\le C\mu(D)+CC_\mu\sum_{k\ge0}(k+1)^22^{-\alpha k}<\infty,
\end{align*}
uniformly in $x$, proving \eqref{eq:log-square-bound}.  If
$2^{-n-1}<\rho\le2^{-n}$, the same decomposition gives
\begin{equation*}
 \sup_{x\in D}\int_{B(x,\rho)\cap D}\ell_x(y)^2\,d\mu(y)
 \le CC_\mu\sum_{k\ge n-1}(k+1)^22^{-\alpha k}
 \longrightarrow0,
\end{equation*}
which is \eqref{eq:log-square-local}.
\end{proof}

By domain monotonicity of the planar Green function,
\begin{equation}
 0\le G_D(x,y)\le C_D+\log^+\frac1{|x-y|}.
 \label{eq:green-log}
\end{equation}
Hence \cref{lem:log-square} gives, for every $\mu\in\mathcal M_*$,
\begin{equation}
 K_G(\mu):=\sup_{x\in D}\int_DG_D(x,y)^2\,d\mu(y)<\infty,
 \qquad
 \iint_{D^2}G_D(x,y)^2\,d\mu(x)d\mu(y)
 \le \mu(D)K_G(\mu)<\infty.
 \label{eq:green-HS}
\end{equation}
For $\mu\in\mathcal M_*$ and $h\in L^2(\mu)$, set
\begin{equation*}
 G_\mu h(x):=\int_DG_D(x,y)h(y)\,d\mu(y).
\end{equation*}
The next lemma identifies this kernel potential with an $H_0^1(D)$ element.

\begin{lemma}
\label{lem:finite-green-potential}
For every $\mu\in\mathcal M_*$ and every real-valued $h\in L^2(\mu)$,
\begin{equation}
 \iint_{D^2}G_D(x,y)|h(x)h(y)|\,d\mu(x)d\mu(y)<\infty.
 \label{eq:hmu-green-energy}
\end{equation}
There is a unique $U_\mu h\in H_0^1(D)$ such that
\begin{equation}
 \E(U_\mu h,v)=\int_Dh\widetilde v\,d\mu,
 \qquad v\in H_0^1(D),
 \label{eq:potential-pairing}
\end{equation}
and
\begin{equation}
 \widetilde{U_\mu h}(x)=G_\mu h(x)
 \quad\text{for quasi-every }x\in D.
 \label{eq:potential-green}
\end{equation}
\end{lemma}

\begin{proof}
By \eqref{eq:green-HS} and Cauchy--Schwarz,
\begin{align*}
 &\iint_{D^2}G_D(x,y)|h(x)h(y)|\,d\mu(x)d\mu(y)\\
 &\qquad\le
 \|G_D\|_{L^2(\mu\otimes\mu)}
 \|h\otimes h\|_{L^2(\mu\otimes\mu)}
 <\infty,
\end{align*}
which proves \eqref{eq:hmu-green-energy}.  Write $h=h_+-h_-$.  Since $\mu$ is
finite, $h_\pm\mu$ are finite positive measures.  For $t>0$ put
\begin{equation*}
 G_t(x,y):=\int_t^\infty q_s^D(x,y)\,ds,
 \qquad
 u_t^\pm(x):=\int_DG_t(x,y)h_\pm(y)\,d\mu(y).
\end{equation*}
The kernel $G_t$ is bounded.  With $P_t=e^{-tL_0}$, set
\begin{equation*}
 g_t^\pm(x):=\int_Dq_t^D(x,y)h_\pm(y)\,d\mu(y).
\end{equation*}
The heat-kernel bound gives $g_t^\pm\in L^\infty(D)\subset L^2(D)$, and
$u_t^\pm=L_0^{-1}g_t^\pm\in\Dom(L_0)$; the displayed kernel integral is
its continuous representative.  For $v\in C_c^\infty(D)$,
\begin{equation}
 \E(u_t^\pm,v)=\int_D(P_tv)h_\pm\,d\mu.
 \label{eq:truncated-potential-pairing}
\end{equation}
Moreover, the semigroup identity gives
\begin{align*}
 &\E(u_t^\pm-u_s^\pm,u_t^\pm-u_s^\pm)\\
 &\quad=
 \iint_{D^2}
 \bigl(G_{2t}+G_{2s}-2G_{t+s}\bigr)(x,y)
 h_\pm(x)h_\pm(y)\,d\mu(x)d\mu(y)
 \longrightarrow0
\end{align*}
as $s,t\downarrow0$, by \eqref{eq:hmu-green-energy} and dominated convergence.
Thus $u_t^\pm$ converges in $H_0^1(D)$ to some $U^\pm$.  Since
$\|P_tv-v\|_\infty\le t\|L_0v\|_\infty$ for
$v\in C_c^\infty(D)$, passing to the limit in
\eqref{eq:truncated-potential-pairing} gives
\begin{equation*}
 \E(U^\pm,v)=\int_Dh_\pm v\,d\mu,
 \qquad
 \left|\int_Dh_\pm v\,d\mu\right|
 \le \E(U^\pm,U^\pm)^{1/2}\E(v,v)^{1/2}.
\end{equation*}
Thus $h_\pm\mu$ has finite energy integral.  Since
$h_\pm\mu\ll\mu$ and $\mu$ is smooth, it charges no polar set.  The
finite-energy potential theorem
\cite[Lemma~2.2.1(ii), Lemma~2.2.3, and Theorem~2.2.5]{FukushimaOshimaTakeda}
extends the pairing to
\begin{equation*}
 \E(U^\pm,v)=\int_Dh_\pm\widetilde v\,d\mu,
 \qquad v\in H_0^1(D).
\end{equation*}
Choose $t_j\downarrow0$ so that $u_{t_j}^\pm$ converges quasi-everywhere to
$\widetilde U^\pm$.  Since $G_t\uparrow G_D$,
\begin{equation*}
 \widetilde U^\pm(x)
 =\int_DG_D(x,y)h_\pm(y)\,d\mu(y)
 \quad\text{for quasi-every }x.
\end{equation*}
Setting $U_\mu h:=U^+-U^-$ proves
\eqref{eq:potential-pairing}--\eqref{eq:potential-green}.  Uniqueness follows
by testing the difference of two solutions against itself.
\end{proof}

\begin{proposition}
\label{prop:green-operator}
For every $\mu\in\mathcal M_*$ the following hold.
\begin{enumerate}[label=\textnormal{(\roman*)}]
\item The operator $G_\mu$ is positive, injective, self-adjoint and
Hilbert--Schmidt on $L^2(\mu)$, and
\begin{equation}
 G_\mu=A_\mu^{-1}.
 \label{eq:green-inverse}
\end{equation}
\item The form inclusion
\begin{equation}
 j_\mu:(V_\mu,\|\cdot\|_{\E,\mu})\longrightarrow L^2(\mu)
 \label{eq:compact-form-inclusion}
\end{equation}
is compact.  In particular, $A_\mu$ has compact resolvent, no zero mode, and
\begin{equation}
 0<\Lambda_1^\mu\le\Lambda_2^\mu\le\cdots\uparrow\infty.
 \label{eq:ordered-spectrum}
\end{equation}
\item
\begin{equation}
 G_\mu:L^2(\mu)\longrightarrow C_b(D),
 \label{eq:green-smoothing}
\end{equation}
so every eigenfunction of $A_\mu$ has a bounded continuous representative.
\end{enumerate}
\end{proposition}

\begin{proof}
By \eqref{eq:green-HS}, $G_\mu$ is Hilbert--Schmidt; symmetry of $G_D$ gives
self-adjointness.  Fix a real-valued $h\in L^2(\mu)$ and let $U_\mu h$ be given by
\cref{lem:finite-green-potential}.  The bound defining $K_G(\mu)$ gives
$G_\mu h\in L^2(\mu)$.  Since $\mu$ is smooth,
\eqref{eq:potential-green} also holds $\mu$-almost everywhere, so
$U_\mu h\in V_\mu$.  Restricting \eqref{eq:potential-pairing} to
$v\in V_\mu$ gives
\begin{equation*}
 A_\mu U_\mu h=h.
\end{equation*}
Conversely, if $u\in\Dom(A_\mu)$ is real-valued and
$w:=u-U_\mu(A_\mu u)\in V_\mu$, then the generator identity and
\eqref{eq:potential-pairing} give
\begin{equation*}
 \E(w,v)=0,\qquad v\in V_\mu.
\end{equation*}
Taking $v=w$ yields $w=0$.  Therefore \eqref{eq:green-inverse} holds.  In addition,
\begin{equation*}
 \langle h,G_\mu h\rangle_{L^2(\mu)}
 =\E(U_\mu h,U_\mu h)\ge0,
\end{equation*}
so $G_\mu$ is positive; injectivity follows from $A_\mu G_\mu=I$.
The identities extend to the complexified Hilbert space by complex linearity.
This proves \textnormal{(i)}.

Choose pairwise disjoint non-empty balls $B_j\Subset D$.  Full support gives
$\mu(B_j)>0$, so the normalized indicators
$\mathbf1_{B_j}/\mu(B_j)^{1/2}$ form an infinite orthonormal family in
$L^2(\mu)$.  The compact spectral theorem applied to the compact injective
operator $G_\mu$ therefore gives
\begin{equation*}
 r_1(\mu)\ge r_2(\mu)\ge\cdots>0,
 \qquad r_n(\mu)\downarrow0.
\end{equation*}
By \eqref{eq:green-inverse},
\begin{equation*}
 \Lambda_n^\mu=r_n(\mu)^{-1}\uparrow\infty,
\end{equation*}
and $0$ is not an eigenvalue of $A_\mu$.  Moreover,
\begin{equation*}
 (A_\mu+1)^{-1}=G_\mu(I+G_\mu)^{-1}
\end{equation*}
is compact.  The first representation theorem also gives
\begin{equation*}
 (A_\mu+1)^{-1}=j_\mu j_\mu^*;
\end{equation*}
the polar decomposition then implies that $j_\mu$ is compact.  This proves
\textnormal{(ii)}.

For \textnormal{(iii)}, let $x_j\to x\in D$.  Fix
$0<\rho<\min\{1/4,\dist(x,\partial D)/4\}$ and take $j$ large enough that
$|x_j-x|<\rho$.  Then $B(x,2\rho)\cap D\subset B(x_j,3\rho)$, so
\eqref{eq:green-log} and \cref{lem:log-square} give
\begin{align*}
 &\int_{B(x,2\rho)\cap D}|G_D(x_j,y)-G_D(x,y)|^2\,d\mu(y)\\
 &\qquad\le
 2\int_{B(x,2\rho)\cap D}G_D(x_j,y)^2\,d\mu(y)
 +2\int_{B(x,2\rho)\cap D}G_D(x,y)^2\,d\mu(y)
 \longrightarrow0
\end{align*}
uniformly as $\rho\downarrow0$.  On $D\setminus B(x,2\rho)$, the poles
remain a fixed distance from $y$, and interior gradient estimates for
$z\mapsto G_D(z,y)$ give
\begin{equation*}
 \sup_{y\in D\setminus B(x,2\rho)}
 |G_D(x_j,y)-G_D(x,y)|
 \le C_{x,\rho}|x_j-x|\longrightarrow0.
\end{equation*}
Hence
\begin{equation*}
 \|G_D(x_j,\cdot)-G_D(x,\cdot)\|_{L^2(\mu)}\longrightarrow0.
\end{equation*}
Together with the definition of $K_G(\mu)$, this yields, for
$h\in L^2(\mu)$,
\begin{align*}
 |G_\mu h(x)|
 &\le K_G(\mu)^{1/2}\|h\|_{L^2(\mu)},\\
 |G_\mu h(x_j)-G_\mu h(x)|
 &\le
 \|G_D(x_j,\cdot)-G_D(x,\cdot)\|_{L^2(\mu)}
 \|h\|_{L^2(\mu)}\longrightarrow0.
\end{align*}
This proves \eqref{eq:green-smoothing}.  Finally, if
$A_\mu\phi=\Lambda\phi$, then
$\phi=\Lambda G_\mu\phi\in C_b(D)$.
\end{proof}

For the conditional slicing arguments, the canonical spectral data must also
depend measurably on the coefficient environment.  For a finite positive Radon
measure $\rho$ and $h\in L^2(\rho)$, write $[h]_\rho$ for its equivalence
class.  For $b\in\Omega_*$, $n\ge1$, and $0<a<c$, write
\begin{equation*}
 \Lambda_n^b:=\Lambda_n^{M_b},
 \qquad
 \Pi_{a,c}^b:=\mathbf1_{(a,c)}(A_b),
\end{equation*}
and set $\Lambda_0^b:=0$ on all of $\Omega$.  Extend the two fields above
by zero off $\Omega_*$:
\begin{equation}
 \Lambda_n^b:=0,
 \qquad
 \Pi_{a,c}^b:=0,
 \qquad b\notin\Omega_*.
 \label{eq:spectral-extensions}
\end{equation}
Recall the rational sine core $Q$ from \eqref{eq:finite-mode-cores}.

\begin{proposition}
\label{prop:borel-spectrum}
For $q\in Q$, rational $0<a<c$, and $n,m\ge1$, the maps
\begin{equation}
 b\longmapsto\Lambda_n^b,
 \qquad
 b\longmapsto
 \Tr\!\left[(\Pi_{a,c}^b\Mult_q\Pi_{a,c}^b)^m\right]
 \label{eq:borel-traces}
\end{equation}
are measurable on $(\Omega,\Borel(\Omega))$.
\end{proposition}

\begin{proof}
Choose a linearly independent sequence
$(h_j)_{j\ge1}\subset C_c^\infty(D)$ whose rational span is uniformly
dense in $C_0(D)$.  Extend the Green field by zero off $\Omega_*$:
\[
 \widehat G_b:=
 \begin{cases}
 G_{M_b},&b\in\Omega_*,\\
 0,&b\notin\Omega_*.
 \end{cases}
\]
For every $i,j$,
\begin{align*}
 &\left\langle[h_i]_{M_b},
 \widehat G_b[h_j]_{M_b}\right\rangle_{L^2(M_b)}\\
 &\qquad=
 \mathbf1_{\Omega_*}(b)
 \lim_{R\to\infty}
 \iint_{D^2}h_i(x)\,(G_D(x,y)\wedge R)\,h_j(y)
 \,dM_b(x)dM_b(y).
\end{align*}
The right-hand side is measurable: $b\mapsto M_b$ is measurable and each
truncated kernel is bounded and measurable.  The measurable-Hilbert-field
construction of \cite[Lemmas~6.19--6.21]{CaiResponseSubmitted} therefore
yields a measurable compact self-adjoint operator field $b\mapsto\widehat G_b$
and measurable eigenvalue maps $b\mapsto r_n(b)$.  On $\Omega_*$,
\begin{equation*}
 \Lambda_n^b=r_n(b)^{-1},
 \qquad
 \Pi_{a,c}^b=\mathbf1_{(c^{-1},a^{-1})}(\widehat G_b).
\end{equation*}
Together with \eqref{eq:spectral-extensions}, this gives measurable
eigenvalue and projection fields on all of $\Omega$.

For $q\in Q$ the scalar maps
\begin{equation*}
 b\longmapsto\int_Dh_iqh_j\,dM_b
\end{equation*}
are measurable.  Applying the same measurable-field lemmas to
$\Pi_{a,c}^b\Mult_q\Pi_{a,c}^b$ and its powers gives
\eqref{eq:borel-traces}.
\end{proof}

\subsection{Conditional coefficient slices}

The finite product density gives an explicit disintegration along every
finite family of coefficient directions.

\begin{proposition}
\label{prop:slices}
Let $f_1,\ldots,f_N\in\mathcal T$ be linearly independent.  There exist a
standard Borel space $Y$, a probability measure $\nu$ on $Y$, a measurable
map
\[
 \Phi:Y\times\R^N\longrightarrow\Omega,
\]
and a probability kernel
\begin{equation*}
 K(\eta,dt)=k(\eta,t)\,dt,
 \qquad k:Y\times\R^N\longrightarrow(0,\infty)
\end{equation*}
with measurable density $k$ such that, for every non-negative measurable
$H:\Omega\to[0,\infty]$,
\begin{equation}
 \int_\Omega H(b)\,\Prob(db)
 =\int_Y\int_{\R^N}H(\Phi(\eta,t))
 K(\eta,dt)\,\nu(d\eta).
 \label{eq:disintegration}
\end{equation}
\end{proposition}

\begin{proof}
Choose $d$ so that each
\[
 d_j:=\delta(f_j)
\]
is supported in the first $d$ coefficient coordinates.  The vectors
$d_1,\ldots,d_N$ are linearly independent because $\delta$ is injective on
$\mathcal T$.  Put
\[
 W:=\Span\{d_1,\ldots,d_N\}\subset\R^d,
 \qquad
 v_t:=\sum_{j=1}^Nt_jd_j,
\]
and
\begin{equation*}
 J_{\rm vol}:=\left[\det(d_i\cdot d_j)_{i,j=1}^N\right]^{1/2}>0.
\end{equation*}
Write $b=(x,\widehat b)$ with $x\in\R^d$, let $\Prob_{\rm tail}$ be the
product law of the remaining coordinates, and decompose uniquely
\[
 x=y+v_t,
 \qquad y\in W^\perp.
\]
With
\[
 p_d(x):=\prod_{k=1}^dp(x_k),
\]
the linear change of variables gives, for every non-negative measurable
$\varphi:\R^d\to[0,\infty]$,
\begin{equation}
 \int_{\R^d}\varphi(x)p_d(x)\,dx
 =\int_{W^\perp}\int_{\R^N}
 \varphi(y+v_t)p_d(y+v_t)J_{\rm vol}\,dt\,
 d\mathcal H^{d-N}(y).
 \label{eq:finite-block-change}
\end{equation}
Set
\begin{equation*}
 Z(y):=\int_{\R^N}p_d(y+v_t)J_{\rm vol}\,dt.
\end{equation*}
Then $Z$ is measurable and positive, while \eqref{eq:finite-block-change}
with $\varphi=1$ gives
\[
 \int_{W^\perp}Z(y)\,d\mathcal H^{d-N}(y)=1.
\]
Hence $Z(y)<\infty$ for $\mathcal H^{d-N}$-almost every $y$.  Define
\[
 Y:=\{(y,\widehat b):0<Z(y)<\infty\}
 \subset W^\perp\times\R^{\N\setminus\{1,\ldots,d\}},
\]
write $\eta=(y,\widehat b)$, and set
\begin{align*}
 \nu(d\eta)
 &:=Z(y)\,d\mathcal H^{d-N}(y)\,
 \Prob_{\rm tail}(d\widehat b),\\
 \Phi(\eta,t)
 &:=(y+v_t,\widehat b),\\
 k(\eta,t)
 &:=\frac{p_d(y+v_t)J_{\rm vol}}{Z(y)}.
\end{align*}
The space $Y$ is standard Borel, $\nu$ is a probability measure, and
$\Phi$ and $k$ are measurable.  Moreover,
\[
 \int_{\R^N}k(\eta,t)\,dt=1,
 \qquad
 k(\eta,t)>0,
\]
so $K(\eta,dt):=k(\eta,t)\,dt$ is a probability kernel.  Formula
\eqref{eq:disintegration} now follows from \eqref{eq:finite-block-change}
after adjoining the independent tail coordinates.
\end{proof}

For the data constructed above,
\begin{equation}
 \Phi(\eta,t)
 =\Phi(\eta,s)
 +\delta\!\left(\sum_{j=1}^N(t_j-s_j)f_j\right),
 \qquad \eta\in Y,\quad s,t\in\R^N.
 \label{eq:fiber-translation}
\end{equation}
In particular, $K(\eta,\cdot)$ is equivalent to Lebesgue measure on
$\R^N$.  Put
\begin{equation*}
 \mathcal G_\eta:=\{t\in\R^N:\Phi(\eta,t)\in\Omega_*\}.
\end{equation*}
Then $\mathcal G_\eta\in\Borel(\R^N)$ for every $\eta\in Y$.

\begin{corollary}
\label{cor:coherent-slices}
For $\nu$-almost every $\eta$,
\begin{equation*}
 K(\eta,\mathcal G_\eta)=1.
\end{equation*}
For any such $\eta$ and any $s\in \mathcal G_\eta$, define
\begin{equation}
 d\widetilde M_{\eta,t}
 =\exp\!\left(\gamma\sum_{j=1}^N(t_j-s_j)f_j\right)
 dM_{\Phi(\eta,s)},
 \qquad t\in\R^N.
 \label{eq:coherent-family}
\end{equation}
Then
\begin{equation}
 \widetilde M_{\eta,t}=M_{\Phi(\eta,t)}
 \qquad(t\in \mathcal G_\eta),
 \label{eq:coherent-agreement}
\end{equation}
and $\widetilde M_{\eta,t}\in\mathcal M_*$ for every $t\in\R^N$.
\end{corollary}

\begin{proof}
Apply \eqref{eq:disintegration} to $H=\mathbf1_{\Omega_*}$.  Since
$\Prob(\Omega_*)=1$,
\[
 1=\int_Y K(\eta,\mathcal G_\eta)\,\nu(d\eta),
\]
and therefore $K(\eta,\mathcal G_\eta)=1$ for $\nu$-almost every $\eta$.  Fix such
an $\eta$ and $s\in \mathcal G_\eta$.  For $t\in \mathcal G_\eta$,
\eqref{eq:fiber-translation} and \cref{cor:canonical-orbit} give
\eqref{eq:coherent-agreement}.  For arbitrary $t$, the measure in
\eqref{eq:coherent-family} is a bounded tilt of
$M_{\Phi(\eta,s)}$ with $\Phi(\eta,s)\in\Omega_*$, and hence belongs to
$\mathcal M_*$.
\end{proof}

The countable core $Q$ will be used below to convert response identities into
identities of finite signed Radon measures.

\begin{lemma}
For every finite signed Radon measure $\sigma$ on $D$,
\begin{equation}
 \int_Dq\,d\sigma=0\quad(q\in Q)
 \qquad\Longrightarrow\qquad
 \sigma=0.
 \label{eq:Q-separates}
\end{equation}
\end{lemma}

\begin{proof}
Let $F\in C_0(D)$ and extend it continuously by zero to $\partial D$.
Its separately odd,
$2$-periodic extension $F^{\rm odd}$ to $\R^2$ satisfies
\[
 F^{\rm odd}(-x_1,x_2)=-F^{\rm odd}(x_1,x_2),
 \qquad
 F^{\rm odd}(x_1,-x_2)=-F^{\rm odd}(x_1,x_2).
\]
Rectangular Fej\'er sums converge uniformly to $F^{\rm odd}$ and, by the
two odd symmetries, restrict on $D$ to finite sums
\[
 T_m(x)=\sum_{r,\ell\le m}c_{r,\ell}^{(m)}
 \sin(r\pi x_1)\sin(\ell\pi x_2)\in\mathcal T,
 \qquad
 \|T_m-F\|_\infty\longrightarrow0.
\]
Approximating the finitely many coefficients by rationals gives
$q_m\in Q$ with $\|q_m-F\|_\infty\to0$.  If $\sigma$ annihilates $Q$,
then
\[
 \left|\int_DF\,d\sigma\right|
 \le\|F-q_m\|_\infty|\sigma|(D)\longrightarrow0.
\]
Thus $\sigma$ annihilates $C_0(D)$ and is zero by the Riesz representation
theorem.
\end{proof}

\section{Pathwise spectral response and transversality}

We study the spectral response under finite-dimensional exponential tilts of a
measure in $\mathcal M_*$.  The resulting first-order formulas will be used
below to obtain splitting and transversality.  Throughout this section, we fix
$\mu\in\mathcal M_*$.

\subsection{Finite-dimensional spectral response}

\begin{lemma}
\label{lem:multipliers}
For $f\in\mathcal T$, let $\Mult_fu:=fu$.  We have
\begin{equation}
 \Mult_f(V_\mu)\subset V_\mu,
 \qquad
 \|\Mult_fu\|_{\E,\mu}
 \le C_f\|u\|_{\E,\mu},
 \qquad u\in V_\mu.
 \label{eq:multiplier-bound}
\end{equation}
For $f_1,\ldots,f_N\in\mathcal T$ and $z\in\C^N$, set
\[
 m_z(x):=\exp\!\left(-\frac{\gamma}{2}\sum_{j=1}^Nz_jf_j(x)\right),
 \qquad
 \Mult_{m_z}u:=m_zu.
\]
We have
\[
 \Mult_{m_z}\in\mathcal B(V_\mu^\C),
 \qquad
 \Mult_{m_z}^{-1}=\Mult_{m_{-z}}.
\]
Moreover,
\[
 \C^N\ni z\longmapsto\Mult_{m_z}\in\mathcal B(V_\mu^\C)
\]
is holomorphic.
\end{lemma}

\begin{proof}
Since $f$ is a finite sine polynomial, $f\in C^\infty(\overline D)$.
Its $C^1$ norm is finite.  For
$u\in H_0^1(D)$,
\[
 \nabla(fu)=f\nabla u+u\nabla f,
\]
and therefore
\begin{align*}
 \E(fu,fu)
 &\le
 \frac1\pi\|f\|_\infty^2\int_D|\nabla u|^2\,dx
 +\frac1\pi\|\nabla f\|_\infty^2\int_D|u|^2\,dx\\
 &\le C_{D,f}\E(u,u),
\end{align*}
where the last step uses the Dirichlet Poincar\'e inequality.  Also
\[
 \|fu\|_{L^2(\mu)}
 \le\|f\|_\infty\|u\|_{L^2(\mu)}.
\]
This proves \eqref{eq:multiplier-bound}.  Hence each
$\Mult_{f_j}$ extends boundedly to $V_\mu^\C$, and
\begin{equation*}
 \Mult_{m_z}
 =\exp\!\left(-\frac{\gamma}{2}\sum_{j=1}^Nz_j\Mult_{f_j}\right)
 =\sum_{r=0}^\infty\frac1{r!}
 \left(-\frac{\gamma}{2}\sum_{j=1}^Nz_j\Mult_{f_j}\right)^r.
\end{equation*}
The series converges locally uniformly in $z$ in operator norm, so
$z\mapsto\Mult_{m_z}$ is holomorphic.  Finally, $m_zm_{-z}=1$, so
$\Mult_{m_z}^{-1}=\Mult_{m_{-z}}$.
\end{proof}

For a positive eigenvalue $\Lambda$ of $A_\mu$, write
\[
 E_\mu(\Lambda):=\ker(A_\mu-\Lambda),
 \qquad
 P_\Lambda^\mu:=\text{the }L^2(\mu)\text{-orthogonal projection onto }
 E_\mu(\Lambda).
\]
For $f\in\mathcal T$ define
\begin{equation*}
 C_f^\Lambda(\mu):=
 P_\Lambda^\mu\Mult_f|_{E_\mu(\Lambda)},
 \qquad
 B_f^\Lambda(\mu):=-\gamma\Lambda C_f^\Lambda(\mu).
\end{equation*}

\begin{proposition}
\label{prop:response}
Let $f_1,\ldots,f_N\in\mathcal T$ and
\[
 d\mu_t=e^{\gamma\sum_{j=1}^Nt_jf_j}\,d\mu,
 \qquad t\in\R^N.
\]
Write $\mathbf e_j$ for the $j$th coordinate vector of $\R^N$.  Then:
\begin{enumerate}[label=\textnormal{(\roman*)}]
\item For every $n\ge1$,
\[
 \R^N\ni t\longmapsto\Lambda_n^{\mu_t}
\]
is continuous.
\item For $t_0\in\R^N$, let $\Lambda$ be an eigenvalue of
$A_{\mu_{t_0}}$ of multiplicity $r$.  For each $1\le j\le N$, there are
real-analytic branches $\lambda_1^{(j)},\ldots,\lambda_r^{(j)}$ of
$A_{\mu_{t_0+s\mathbf e_j}}$ through $\Lambda$ such that
\begin{equation*}
 \lambda_a^{(j)}(0)=\Lambda,
 \qquad 1\le a\le r.
\end{equation*}
The $r$ numbers $(\lambda_a^{(j)})'(0)$ are precisely the eigenvalues of
$B_{f_j}^\Lambda(\mu_{t_0})$, each repeated according to its multiplicity.
If
$\phi_1,\ldots,\phi_r$ is a real $L^2(\mu_{t_0})$-orthonormal basis of
$E_{\mu_{t_0}}(\Lambda)$, then
\begin{equation}
 \left\langle
 B_{f_j}^\Lambda(\mu_{t_0})\phi_a,\phi_c
 \right\rangle_{L^2(\mu_{t_0})}
 =-\gamma\Lambda\int_D f_j\phi_a\phi_c\,d\mu_{t_0},
 \qquad 1\le a,c\le r.
 \label{eq:cluster-response}
\end{equation}
\end{enumerate}
\end{proposition}

\begin{proof}
Fix $t_0\in\R^N$.  Since $\mu_{t_0}\in\mathcal M_*$, replacing $\mu$ by
$\mu_{t_0}$ and translating the parameter reduces the local assertions to
$t_0=0$.  Since $\sum_jt_jf_j$ is bounded,
\[
 V_{\mu_t}=V_\mu,
 \qquad
 \|\cdot\|_{\E,\mu_t}\asymp\|\cdot\|_{\E,\mu}
\]
locally uniformly in $t$.  The unitary map
\[
 U_t:L^2(\mu_t)\to L^2(\mu),
 \qquad
 U_tu=e^{\gamma\sum_jt_jf_j/2}u,
\]
transports $A_{\mu_t}$ to the common-domain form
\begin{equation*}
 a_t(u,v)=\E(m_tu,m_tv),
 \qquad \Dom a_t=V_\mu,
\end{equation*}
with $m_t$ as in \cref{lem:multipliers}.  We use complexified inner
products and forms that are linear in the first variable.  On the complexified
form space set
\begin{equation*}
 a_z(u,v):=\E(m_zu,m_{\overline z}v),
 \qquad \Dom a_z=V_\mu^\C.
\end{equation*}
By \cref{lem:multipliers,prop:green-operator} and
\cite[Propositions~2.2 and 2.5]{CaiResponseSubmitted}, $(a_z)$ is a
holomorphic type~(a) family with compact resolvent.  Hence
\cite[Lemma~2.6]{CaiResponseSubmitted} gives part~(i).

For the $j$th coordinate direction,
\begin{equation}
 \partial_{t_j}a_t(u,v)\big|_{t=0}
 =-\frac\gamma2\bigl[\E(f_ju,v)+\E(u,f_jv)\bigr].
 \label{eq:form-derivative}
\end{equation}
If $u,v\in E_\mu(\Lambda)$, then
\begin{equation*}
 \partial_{t_j}a_t(u,v)\big|_{t=0}
 =-\gamma\Lambda\int_Df_juv\,d\mu
 =\left\langle B_{f_j}^\Lambda(\mu)u,v\right\rangle_{L^2(\mu)}.
\end{equation*}
Thus \eqref{eq:cluster-response} holds, and
\cite[Proposition~3.2]{CaiResponseSubmitted} identifies the spectrum of this
compression with the derivatives of the analytic branches.  This proves
part~(ii).
\end{proof}

\begin{corollary}
\label{cor:simple-response}
Let $f_1,\ldots,f_N$ and $\mu_t$ be as in \cref{prop:response}.  Fix
$t_0\in\R^N$ and $n\ge1$.  Suppose that $\Lambda_n^{\mu_{t_0}}$ has
multiplicity one.  Then there is a neighborhood $U$ of $t_0$ such that
$\Lambda_n^{\mu_t}$ has multiplicity one for every $t\in U$, and the function
$t\mapsto\Lambda_n^{\mu_t}$ is real-analytic on $U$.  Choose a real
eigenfunction $\phi_{n,t_0}$ satisfying
\begin{equation*}
 A_{\mu_{t_0}}\phi_{n,t_0}
 =\Lambda_n^{\mu_{t_0}}\phi_{n,t_0},
 \qquad
 \|\phi_{n,t_0}\|_{L^2(\mu_{t_0})}=1.
\end{equation*}
Then, for $1\le j\le N$,
\begin{equation}
 \partial_{t_j}\Lambda_n^{\mu_t}\big|_{t=t_0}
 =-\gamma\Lambda_n^{\mu_{t_0}}
 \int_Df_j\phi_{n,t_0}^2\,d\mu_{t_0}.
 \label{eq:simple-response}
\end{equation}
\end{corollary}

\begin{proof}
Part~(i) of \cref{prop:response} preserves an isolating spectral gap on
some neighborhood $U$ of $t_0$.  The corresponding rank-one eigenvalue branch
is real-analytic by \cite[Proposition~2.5]{CaiResponseSubmitted}, and
\cref{prop:response}(ii) gives \eqref{eq:simple-response}.
\end{proof}

\subsection{First-order splitting}

We now use the countable core $Q$ from \eqref{eq:finite-mode-cores} to find
finite-mode directions that split multiple eigenvalues to first order.  Recall
that, for an eigenvalue $\Lambda$ of $A_\mu$ and $q\in Q$,
\begin{equation*}
 E_\mu(\Lambda)=\ker(A_\mu-\Lambda),
 \qquad
 C_q^\Lambda(\mu)=P_\Lambda^\mu\Mult_q|_{E_\mu(\Lambda)}.
\end{equation*}
For orthonormal $u,v\in E_\mu(\Lambda)$,
\begin{equation}
 \left\langle C_q^\Lambda(\mu)u,v\right\rangle_{L^2(\mu)}
 =\int_Dquv\,d\mu,
 \qquad
 \left\langle C_q^\Lambda(\mu)u,u\right\rangle_{L^2(\mu)}
 -\left\langle C_q^\Lambda(\mu)v,v\right\rangle_{L^2(\mu)}
 =\int_Dq(u^2-v^2)\,d\mu.
 \label{eq:two-dimensional-response}
\end{equation}
Thus a direction is non-scalar on the plane spanned by $u,v$ exactly when the
two quantities on the right of \eqref{eq:two-dimensional-response} do not both
vanish.

\begin{lemma}
\label{lem:non-scalar-compression}
Let $\Lambda$ be an eigenvalue of $A_\mu$ with
$\dim E_\mu(\Lambda)\ge2$.  For every orthonormal pair
$u,v\in E_\mu(\Lambda)$, there exists $q\in Q$ such that
\begin{equation}
 \left(
  \int_Dquv\,d\mu,
  \int_Dq(u^2-v^2)\,d\mu
 \right)\ne(0,0).
 \label{eq:non-scalar-compression}
\end{equation}
\end{lemma}

\begin{proof}
If both integrals vanished for every $q\in Q$, then
\eqref{eq:Q-separates} would give
\begin{equation*}
 uv\,\mu=0,
 \qquad
 (u^2-v^2)\mu=0.
\end{equation*}
Hence $uv=0$ and $u^2=v^2$ $\mu$-a.e., so $u=v=0$ $\mu$-a.e., contrary to
$\|u\|_{L^2(\mu)}=1$.
\end{proof}

The preceding plane-by-plane separation has the following finite-dimensional
consequence for the already defined response operators.

\begin{lemma}
\label{lem:finite-dimensional-splitting}
Let $\Lambda$ be an eigenvalue of $A_\mu$ and put
\[
 r:=\dim E_\mu(\Lambda)\ge2.
\]
Then there exist $q_1,\ldots,q_m\in Q$ and $a_1,\ldots,a_m\in\R$ such that
\begin{equation}
 \sum_{j=1}^m a_j C_{q_j}^\Lambda(\mu)
 \quad\text{has }r\text{ distinct eigenvalues}.
 \label{eq:real-simple-combination}
\end{equation}
\end{lemma}

\begin{proof}
Choose a finite real linear combination of the operators
$C_q^\Lambda(\mu)$ having the largest possible number of distinct eigenvalues,
and denote it by $C_{\max}$.  If $C_{\max}$ had fewer than $r$ distinct eigenvalues,
some eigenvalue $\alpha$ would have an eigenspace
\begin{equation*}
 E_\alpha:=\ker(C_{\max}-\alpha I),
 \qquad s:=\dim E_\alpha\ge2.
\end{equation*}
Choose orthonormal $u,v\in E_\alpha$.  By
\cref{lem:non-scalar-compression}, there is $q\in Q$ for which
\begin{equation*}
 \left(
  \langle C_q^\Lambda(\mu)u,v\rangle_{L^2(\mu)},
  \langle C_q^\Lambda(\mu)u,u\rangle_{L^2(\mu)}
  -\langle C_q^\Lambda(\mu)v,v\rangle_{L^2(\mu)}
 \right)\ne(0,0).
\end{equation*}
Extend $u,v$ to an orthonormal basis $w_1,\ldots,w_s$ of
$E_\alpha$.  The symmetric matrix
\begin{equation*}
 \left(
  \langle C_q^\Lambda(\mu)w_i,w_j\rangle_{L^2(\mu)}
 \right)_{i,j=1}^s
\end{equation*}
is therefore non-scalar.  Let
$\xi_1\le\cdots\le\xi_s$ be its eigenvalues; then
$\xi_p<\xi_{p+1}$ for some $p$.  The finite-dimensional first-order
perturbation formula gives branches issuing from $\alpha$ with
\begin{equation*}
 \lambda_i(\varepsilon)
 =\alpha+\varepsilon\xi_i+O(\varepsilon^2),
 \qquad 1\le i\le s.
\end{equation*}
Hence
\begin{equation*}
 \lambda_{p+1}(\varepsilon)-\lambda_p(\varepsilon)
 =\varepsilon(\xi_{p+1}-\xi_p)+O(\varepsilon^2)>0
\end{equation*}
for all sufficiently small $\varepsilon>0$.  The other spectral clusters of
$C_{\max}$ remain separated, so
$C_{\max}+\varepsilon C_q^\Lambda(\mu)$ has more distinct eigenvalues than
$C_{\max}$, a contradiction.  Therefore $C_{\max}$ has $r$
distinct eigenvalues, which is \eqref{eq:real-simple-combination}.
\end{proof}

We now return from real linear combinations to a single direction in the
countable core $Q$.

\begin{proposition}
\label{prop:single-splitting}
Every multiple positive eigenvalue $\Lambda$ of $A_\mu$ admits $q\in Q$ such
that
\begin{equation}
 B_q^\Lambda(\mu)
 \quad\text{has }\dim E_\mu(\Lambda)\text{ distinct eigenvalues}.
 \label{eq:single-splitting-direction}
\end{equation}
\end{proposition}

\begin{proof}
By \cref{lem:finite-dimensional-splitting}, choose
$q_1,\ldots,q_m\in Q$ and $a_1,\ldots,a_m\in\R$ such that
\begin{equation*}
 \sum_{j=1}^m a_j C_{q_j}^\Lambda(\mu)
\end{equation*}
has distinct eigenvalues.  Its discriminant is therefore nonzero and
depends continuously on the coefficients.  The same remains true after
replacing $a_j$ by sufficiently close rational numbers $b_j$.  Set
\begin{equation*}
 q:=\sum_{j=1}^m b_jq_j\in Q.
\end{equation*}
By linearity in the direction,
\begin{equation*}
 C_q^\Lambda(\mu)
 =\sum_{j=1}^m b_jC_{q_j}^\Lambda(\mu),
 \qquad
 B_q^\Lambda(\mu)=-\gamma\Lambda C_q^\Lambda(\mu),
\end{equation*}
so \eqref{eq:single-splitting-direction} follows.
\end{proof}

The same choice can be made simultaneously for finitely many multiple
clusters.

\begin{corollary}
\label{cor:simultaneous-splitting}
For any finite collection
$\Lambda^{(1)},\ldots,\Lambda^{(L)}$ of distinct multiple positive
eigenvalues, there
exists $q\in Q$ such that, for every $1\le\ell\le L$,
\begin{equation}
 B_q^{\Lambda^{(\ell)}}(\mu)
 \quad\text{has }\dim E_\mu(\Lambda^{(\ell)})\text{ distinct eigenvalues}.
 \label{eq:simultaneous-splitting}
\end{equation}
\end{corollary}

\begin{proof}
For each $\ell$, choose $q^{(\ell)}\in Q$ from
\cref{prop:single-splitting} so that
$B_{q^{(\ell)}}^{\Lambda^{(\ell)}}(\mu)$ has distinct eigenvalues.  For
$a=(a_1,\ldots,a_L)\in\R^L$, put
\begin{equation*}
 q_a:=\sum_{\ell=1}^L a_\ell q^{(\ell)}\in\mathcal T.
\end{equation*}
For each fixed $\ell$, the discriminant of
$B_{q_a}^{\Lambda^{(\ell)}}(\mu)$ is a polynomial in $a$.  It is not the
zero polynomial, since it is nonzero when $a_\ell=1$ and $a_k=0$ for
$k\ne\ell$.  Hence the coefficient vectors $a$ for which every
$B_{q_a}^{\Lambda^{(\ell)}}(\mu)$, $1\le\ell\le L$, has distinct
eigenvalues form a nonempty open dense subset of $\R^L$.  Choose
$a\in\Q^L$ in this set.  Then $q_a\in Q$ and \eqref{eq:simultaneous-splitting}
holds for every $\ell$.
\end{proof}

\subsection{The square identity and transversality}

Recall $L_0=-(2\pi)^{-1}\Delta_D$, and set
\begin{equation*}
 \Gamma(u,v):=\frac1{2\pi}\nabla u\cdot\nabla v,
 \qquad
 \Gamma(u):=\Gamma(u,u).
\end{equation*}
For $u\in H_0^1(D)$ we understand $L_0u$ distributionally through
\[
 \langle L_0u,\zeta\rangle=\E(u,\zeta),
 \qquad \zeta\in C_c^\infty(D).
\]
The square identity separates a singular response measure from an absolutely
continuous energy measure.

\begin{lemma}
\label{lem:square}
Let $\phi$ be real-valued and satisfy
$A_\mu\phi=\Lambda\phi$ with $\Lambda>0$ and
$\|\phi\|_{L^2(\mu)}=1$.  Then $\phi^2\in H_0^1(D)$ and, in
$\mathcal D'(D)$,
\begin{equation}
 L_0(\phi^2)=2\Lambda\phi^2\mu-2\Gamma(\phi)\,dx.
 \label{eq:square-identity}
\end{equation}
The two measures on the right satisfy
\begin{equation*}
 2\Lambda\phi^2\mu\perp dx,
 \qquad
 2\Gamma(\phi)\,dx\ll dx,
 \qquad
 (2\Lambda\phi^2\mu)(D)=(2\Gamma(\phi)\,dx)(D)=2\Lambda.
\end{equation*}
\end{lemma}

\begin{proof}
By \cref{prop:green-operator}, $\phi\in H_0^1(D)\cap L^\infty(D)$.  The Sobolev
chain rule gives
\begin{equation}
 \phi^2\in H_0^1(D),
 \qquad
 \nabla(\phi^2)=2\phi\nabla\phi.
 \label{eq:square-chain}
\end{equation}
For real $\zeta\in C_c^\infty(D)$, the function $\zeta\phi$ belongs to the
trace domain.  Hence
\begin{align}
 \Lambda\int_D\zeta\phi^2\,d\mu
 &=\E(\phi,\zeta\phi)\notag\\
 &=\int_D\zeta\Gamma(\phi)\,dx
   +\int_D\phi\Gamma(\phi,\zeta)\,dx,
 \label{eq:eigen-test-product}
\end{align}
while \eqref{eq:square-chain} gives
\begin{equation}
 \E(\phi^2,\zeta)
 =2\int_D\phi\Gamma(\phi,\zeta)\,dx.
 \label{eq:square-test}
\end{equation}
Eliminating the last integral between \eqref{eq:eigen-test-product} and
\eqref{eq:square-test} yields
\[
 \E(\phi^2,\zeta)
 =2\Lambda\int_D\zeta\phi^2\,d\mu
  -2\int_D\zeta\Gamma(\phi)\,dx,
\]
which is \eqref{eq:square-identity}.  Since $\mu\perp dx$ and
$\Gamma(\phi)\ge0$, the two measures in \eqref{eq:square-identity} are
mutually singular.  Their masses are
\begin{equation*}
 2\Lambda\int_D\phi^2\,d\mu=2\Lambda,
 \qquad
 2\int_D\Gamma(\phi)\,dx=2\E(\phi,\phi)=2\Lambda.
\end{equation*}
\end{proof}

For a positive eigenvalue $\Lambda_n^\mu$ of multiplicity one, choose a
normalized real eigenfunction $\phi_n^\mu$ and define the real-linear
functional
\begin{equation*}
 \ell_n^\mu(f)
 :=-\gamma\Lambda_n^\mu
 \int_Df(\phi_n^\mu)^2\,d\mu,
 \qquad f\in\mathcal T.
\end{equation*}
This value is independent of the sign of $\phi_n^\mu$.
By \eqref{eq:simple-response},
\begin{equation*}
 \ell_n^\mu(f)
 =\left.\frac{d}{dt}\Lambda_n^{e^{\gamma tf}\mu}\right|_{t=0}.
\end{equation*}
The square identity turns any linear response relation into a Vandermonde
system.

\begin{proposition}
\label{prop:transversality}
Let $\Lambda_{n_1}^\mu,\ldots,\Lambda_{n_N}^\mu$ be pairwise distinct
positive eigenvalues, each of multiplicity one.  Then the response functionals
$\ell_{n_1}^\mu,\ldots,\ell_{n_N}^\mu$ are linearly independent on $Q$.
\end{proposition}

\begin{proof}
Write $\Lambda_i=\Lambda_{n_i}^\mu$ and
$\phi_i=\phi_{n_i}^\mu$.  Suppose
\begin{equation}
 \sum_{i=1}^Na_i\ell_{n_i}^\mu(q)=0
 \qquad(q\in Q).
 \label{eq:response-relation}
\end{equation}
By \eqref{eq:Q-separates}, \eqref{eq:response-relation} implies
\begin{equation*}
 \sum_{i=1}^Na_i\Lambda_i\phi_i^2\,\mu=0.
\end{equation*}
Since $\mu$ has full support and the $\phi_i$ are continuous,
\begin{equation}
 u_1:=\sum_{i=1}^Na_i\Lambda_i\phi_i^2=0
 \qquad\text{on }D.
 \label{eq:u1-zero}
\end{equation}
For $k\ge1$ put
\[
 u_k:=\sum_{i=1}^Na_i\Lambda_i^k\phi_i^2.
\]
If $u_k=0$, then \cref{lem:square} gives
\begin{equation*}
 0=L_0u_k
 =2u_{k+1}\,\mu
  -2\sum_{i=1}^Na_i\Lambda_i^k\Gamma(\phi_i)\,dx.
\end{equation*}
The first measure on the right is singular with respect to $dx$, whereas the
second is absolutely continuous.  Uniqueness of the Lebesgue decomposition
forces $u_{k+1}\mu=0$.  Full support and continuity then give
$u_{k+1}=0$ on $D$.  Starting from \eqref{eq:u1-zero}, induction yields
\begin{equation}
 \sum_{i=1}^Na_i\Lambda_i^k\phi_i(x)^2=0,
 \qquad 1\le k\le N,\quad x\in D.
 \label{eq:vandermonde-system}
\end{equation}
For fixed $x$, the coefficient matrix in \eqref{eq:vandermonde-system} has
determinant
\begin{equation*}
 \det(\Lambda_i^k)_{k,i=1}^N
 =\left(\prod_{i=1}^N\Lambda_i\right)
  \prod_{1\le i<j\le N}(\Lambda_j-\Lambda_i)\ne0.
\end{equation*}
Hence $a_i\phi_i(x)^2=0$ for every $i$ and every $x$.  Integrating against
$\mu$ and using $\|\phi_i\|_{L^2(\mu)}=1$ gives $a_i=0$ for every $i$.
\end{proof}

\section{Conditional slicing and almost-sure simplicity}

A splitting direction makes the relevant collisions isolated along each
coherent one-dimensional orbit.  We first encode this property by measurable finite-window events and then
integrate against the conditional density from \cref{prop:slices}.  Recall from \eqref{eq:finite-mode-cores} that
$Q=\Span_\Q\{e_k:k\ge1\}$ is the countable rational sine core.

\subsection{Local splitting and measurable collision events}

For a measure $\mu\in\mathcal M_*$ and $q\in Q$, we say that $q$ splits a
multiple positive eigenvalue $\Lambda$ to first order when
$B_q^\Lambda(\mu)$ has pairwise distinct eigenvalues.

Fix $\mu_0\in\mathcal M_*$ and $K\ge1$, and suppose that $q\in Q$ splits to
first order every multiple positive eigenvalue of $A_{\mu_0}$ whose full
ordered multiplicity block contains one of the gaps $j,j+1$, $1\le j\le K$.
Put
\[
 d\mu_t=e^{\gamma tq}\,d\mu_0.
\]
The corresponding local consequence is the following.

\begin{lemma}
\label{lem:local-isolation}
There exists $\delta>0$ such that
\begin{equation}
 0<|t|<\delta
 \quad\Longrightarrow\quad
 \Lambda_j^{\mu_t}<\Lambda_{j+1}^{\mu_t},
 \qquad 1\le j\le K.
 \label{eq:local-split-isolation}
\end{equation}
\end{lemma}

\begin{proof}
Only finitely many full multiplicity blocks of $A_{\mu_0}$ meet the labels
$1,\ldots,K+1$.  Denote their distinct eigenvalues by
$\Lambda^{(1)},\ldots,\Lambda^{(L)}$, and choose pairwise disjoint bounded
intervals $I_1,\ldots,I_L$ such that
\begin{equation*}
 \overline I_\ell\cap\Spec(A_{\mu_0})
 =\{\Lambda^{(\ell)}\},
 \qquad 1\le \ell\le L.
\end{equation*}
The common-domain analytic perturbation from \cref{prop:response} and the
Riesz-projection stability in
\cite[Proposition~2.5]{CaiResponseSubmitted} imply that, for some
$\delta_0>0$,
\begin{equation*}
 \operatorname{rank}\mathbf1_{I_\ell}(A_{\mu_t})
 =\operatorname{rank}\mathbf1_{I_\ell}(A_{\mu_0}),
 \qquad |t|<\delta_0,\quad 1\le \ell\le L.
\end{equation*}
Thus eigenvalues issued from distinct blocks stay separated.

Consider a gap $j,j+1$ lying inside one block of multiplicity $r$ at
$\Lambda$.  By \cref{prop:response} there are analytic branches
$\lambda_1,\ldots,\lambda_r$ with
\[
 \lambda_i(t)=\Lambda+t\lambda_i'(0)+o(t).
\]
The first-order splitting hypothesis says that the branch derivatives are
pairwise distinct.  Hence
\[
 d_*:=\min_{i<\ell}|\lambda_i'(0)-\lambda_\ell'(0)|>0,
\]
and, after shrinking the parameter interval,
\begin{equation*}
 |\lambda_i(t)-\lambda_\ell(t)|
 \ge \frac{d_*}{2}|t|>0,
 \qquad i\ne\ell,\quad0<|t|<\delta_\Lambda.
\end{equation*}
Therefore every gap internal to this block opens for small nonzero $t$.
Taking the minimum of the finitely many radii
$\delta_0,\delta_\Lambda$ proves \eqref{eq:local-split-isolation}.
\end{proof}

For $K\ge1$ set
\begin{equation*}
 \mathcal C_K
 :=
 \Omega_*\cap
 \bigcup_{j=1}^K\{\Lambda_j^b=\Lambda_{j+1}^b\}.
\end{equation*}
By \cref{prop:borel-spectrum}, $\mathcal C_K\in\Borel(\Omega)$.  For
$q\in Q\setminus\{0\}$ let $\mathcal C_{K,q}$ be the subset of
$\mathcal C_K$ on which $q$ splits to first order every full multiplicity
block containing at least one of the first $K$ gaps.  The splitting condition
can be read from rational spectral windows.

\begin{lemma}
\label{lem:measurable-splitting}
For every $K\ge1$ and $q\in Q\setminus\{0\}$,
\begin{equation*}
 \mathcal C_{K,q}\in\Borel(\Omega).
\end{equation*}
\end{lemma}

\begin{proof}
For $1\le p<\ell$, define the full-block event
\begin{equation*}
 \mathcal H_{p,\ell}^{\rm block}
 :=
 \Omega_*\cap
 \{\Lambda_{p-1}^b<\Lambda_p^b=\cdots=\Lambda_\ell^b
 <\Lambda_{\ell+1}^b\}.
\end{equation*}
For rational $0<a<c$, set
\begin{equation*}
 \mathcal W_{p,\ell}^{a,c}
 :=
 \mathcal H_{p,\ell}^{\rm block}\cap
 \{\Lambda_{p-1}^b<a<\Lambda_p^b=\Lambda_\ell^b
 <c<\Lambda_{\ell+1}^b\}.
\end{equation*}
By \cref{prop:borel-spectrum},
$\mathcal H_{p,\ell}^{\rm block},\mathcal W_{p,\ell}^{a,c}
\in\Borel(\Omega)$.  On $\mathcal W_{p,\ell}^{a,c}$ the projection
\[
 \Pi_{a,c}^b=\mathbf1_{(a,c)}(A_b)
\]
has rank
\[
 d:=\ell-p+1.
\]
Its compression of $\Mult_q$ is $C_q^{\Lambda_p^b}(M_b)$, while
$B_q^{\Lambda_p^b}(M_b)$ is the nonzero scalar multiple
$-\gamma\Lambda_p^b C_q^{\Lambda_p^b}(M_b)$.  Thus $q$ splits this block
exactly when the compression below has nonzero discriminant.

For all $b\in\Omega$ define
\begin{equation*}
 \Delta_{p,\ell;a,c,q}(b)
 :=
 \begin{cases}
 \Disc\!\left(
 \Pi_{a,c}^b\Mult_q\Pi_{a,c}^b
 \big|_{\operatorname{Ran}\Pi_{a,c}^b}
 \right),
 & b\in\Omega_*,\ \operatorname{rank}\Pi_{a,c}^b=d,\\[1mm]
 0,&\text{otherwise}.
 \end{cases}
\end{equation*}
The rank map is measurable because
\begin{equation*}
 \operatorname{rank}\Pi_{a,c}^b
 =\sum_{n\ge1}\mathbf1_{(a,c)}(\Lambda_n^b).
\end{equation*}
On the rank-$d$ locus, Newton's identities express the characteristic
polynomial, hence the discriminant, as a polynomial in the power sums
\begin{equation*}
 \Tr\!\left[
 (\Pi_{a,c}^b\Mult_q\Pi_{a,c}^b)^m
 \right],
 \qquad 1\le m\le d.
\end{equation*}
By \cref{prop:borel-spectrum}, these power sums are measurable.
It follows that
$\Delta_{p,\ell;a,c,q}:\Omega\to\R$ is measurable.

Define
\begin{equation*}
 \mathcal S_{p,\ell}(q)
 :=
 \bigcup_{\substack{a,c\in\Q_{>0}\\a<c}}
 \left(
 \mathcal W_{p,\ell}^{a,c}
 \cap\{\Delta_{p,\ell;a,c,q}\ne0\}
 \right).
\end{equation*}
If $b\in\mathcal H_{p,\ell}^{\rm block}$, rational endpoints $a,c$ can be
chosen strictly between the neighboring eigenvalues.  Therefore
$b\in\mathcal S_{p,\ell}(q)$ if and only if $q$ splits that full block to
first order.  Consequently
\begin{equation*}
 \mathcal C_{K,q}
 =
 \mathcal C_K\cap
 \bigcap_{p=1}^K\bigcap_{\ell=p+1}^{\infty}
 \left(
 (\mathcal H_{p,\ell}^{\rm block})^c
 \cup\mathcal S_{p,\ell}(q)
 \right).
\end{equation*}
Every set on the right belongs to $\Borel(\Omega)$, so
$\mathcal C_{K,q}\in\Borel(\Omega)$.
\end{proof}

The simultaneous splitting corollary gives the countable cover
\begin{equation}
 \mathcal C_K
 =
 \bigcup_{q\in Q\setminus\{0\}}\mathcal C_{K,q}.
 \label{eq:collision-cover}
\end{equation}
Indeed, only finitely many full multiplicity blocks meet the first $K+1$
ordered eigenvalues, and \cref{cor:simultaneous-splitting} supplies one
direction that splits all multiple blocks among them.  The reverse inclusion
is immediate.  It remains to evaluate each fixed-direction event on the
corresponding one-dimensional conditional fibers.

\begin{lemma}
\label{lem:fixed-direction-null}
For every $K\ge1$ and $q\in Q\setminus\{0\}$,
\begin{equation}
 \Prob(\mathcal C_{K,q})=0.
 \label{eq:fixed-direction-null}
\end{equation}
\end{lemma}

\begin{proof}
Use the one-dimensional conditional decomposition from
\cref{prop:slices} for the direction $q$, and write its data as
$Y_q,\nu_q,\Phi_q,K_q$, with
\begin{equation*}
 K_q(\eta,dt)=k_q(\eta,t)\,dt.
\end{equation*}
For $\nu_q$-almost every $\eta$, let $\mathcal G_\eta$ and
$(\widetilde M_{\eta,t})_{t\in\R}$ be the structural coherent family from
\cref{cor:coherent-slices}.  Define
\begin{equation*}
 Z_\eta:=\{t\in\R:\Phi_q(\eta,t)\in\mathcal C_{K,q}\}.
\end{equation*}
By \cref{lem:measurable-splitting}, $Z_\eta\in\Borel(\R)$.  Moreover,
$\mathcal C_{K,q}\subset\Omega_*$, so $Z_\eta\subset\mathcal G_\eta$
and the canonical and coherent measures agree at every point of $Z_\eta$.

Fix $t_0\in Z_\eta$.  Coherence gives
\begin{equation*}
 d\widetilde M_{\eta,t_0+s}
 =e^{\gamma sq}\,d\widetilde M_{\eta,t_0}.
\end{equation*}
At $\widetilde M_{\eta,t_0}$ the direction $q$ splits every multiple block
responsible for one of the first $K$ collisions.  Hence
\cref{lem:local-isolation} gives $\delta(t_0)>0$ such that
\[
 0<|s|<\delta(t_0)
 \quad\Longrightarrow\quad
 \Lambda_j^{\widetilde M_{\eta,t_0+s}}
 <
 \Lambda_{j+1}^{\widetilde M_{\eta,t_0+s}}
 \quad(1\le j\le K).
\]
If $t_0+s\in Z_\eta$, the coherent and canonical measures agree there,
so the preceding strict inequalities exclude every nonzero
$|s|<\delta(t_0)$.  Therefore
\begin{equation*}
 (t_0-\delta(t_0),t_0+\delta(t_0))\cap Z_\eta=\{t_0\}.
\end{equation*}
Every point of $Z_\eta$ is isolated.  Since $\R$ is second countable,
$Z_\eta$ is countable, and hence
\begin{equation*}
 K_q(\eta,Z_\eta)=0
\end{equation*}
for almost every $\eta$.  Therefore
\begin{equation*}
 \Prob(\mathcal C_{K,q})
 =\int_{Y_q}K_q(\eta,Z_\eta)\,\nu_q(d\eta)=0,
\end{equation*}
which proves \eqref{eq:fixed-direction-null}.
\end{proof}

\subsection{Proof of the simplicity assertion}

Fix $K\ge1$.  By \eqref{eq:collision-cover}, countability of $Q$, and
\cref{lem:fixed-direction-null},
\[
 \Prob(\mathcal C_K)
 \le
 \sum_{q\in Q\setminus\{0\}}\Prob(\mathcal C_{K,q})=0.
\]
Nonsimplicity on $\Omega_*$ is the event
\[
 \bigcup_{K\ge1}\mathcal C_K.
\]
Since $\Prob(\Omega_*)=1$,
\[
 0<\Lambda_1^b<\Lambda_2^b<\cdots
 \qquad\text{almost surely}.
\]
The divergence $\Lambda_n^b\uparrow\infty$ and compact resolvent follow from
\cref{prop:green-operator}.

\section{Joint eigenvalue densities by response submersion}

\subsection{Measurable submersion events and fiberwise nullity}

Fix $N\ge1$ and $1\le n_1<\cdots<n_N$.  On the almost-sure simplicity
event, \cref{prop:transversality} makes the selected response functionals
linearly independent.  We choose finitely many coefficient directions with a
nonzero response determinant and apply the inverse function theorem on the
corresponding conditional slices.

Using the all-space measurable extensions from
\cref{prop:borel-spectrum}, define
\begin{equation*}
 \boldsymbol\Lambda(b)
 :=
 (\Lambda_{n_1}^b,\ldots,\Lambda_{n_N}^b),
 \qquad b\in\Omega.
\end{equation*}
To pass from functional independence to a finite-dimensional coefficient
slice, we need finitely many directions.

\begin{lemma}
\label{lem:finite-witness}
Let $\ell_1,\ldots,\ell_N:\mathcal T\to\R$ be real-linear
functionals, and assume that their restrictions to $Q$ are linearly
independent over $\R$.  There exist directions $q_1,\ldots,q_N\in Q$
satisfying
\begin{equation}
 \det(\ell_i(q_j))_{i,j=1}^N\ne0.
 \label{eq:finite-witness}
\end{equation}
At the same time, the chosen directions are linearly independent over $\R$
in $\mathcal T$.
\end{lemma}

\begin{proof}
Define
\[
 T:\mathcal T\longrightarrow\R^N,\qquad
 T(f):=(\ell_1(f),\ldots,\ell_N(f)).
\]
If $\Span_\R T(Q)$ were a proper subspace of $\R^N$, there would exist
$a=(a_1,\ldots,a_N)\ne0$ orthogonal to it.  Then
\[
 \sum_{i=1}^Na_i\ell_i(q)=0
 \qquad(q\in Q),
\]
contradicting the assumed independence of the restrictions.  Hence $T(Q)$
spans $\R^N$, and one can choose $q_1,\ldots,q_N\in Q$ such that
$T(q_1),\ldots,T(q_N)$ form a basis.  This is
\eqref{eq:finite-witness}.  If $\sum_jc_jq_j=0$ in $\mathcal T$ with
$c_j\in\R$, real-linearity gives $\sum_jc_jT(q_j)=0$; since these vectors
form a basis, every $c_j$ vanishes.  Thus the chosen directions are also
real-linearly independent in $\mathcal T$.
\end{proof}

Let $\mathcal Q_N$ denote the countable set of $N$-tuples
\[
 \mathbf q=(q_1,\ldots,q_N)\in Q^N
\]
that are linearly independent over $\R$ in $\mathcal T$.
For such a tuple define
\begin{equation*}
 \mathcal S_{\mathbf q}
 :=
 \left\{
 b\in\Omega_*:
 \begin{array}{l}
 \Lambda_{n_1}^b,\ldots,\Lambda_{n_N}^b
 \text{ have multiplicity one},\\[1mm]
 \det\bigl(\ell_{n_i}^{M_b}(q_j)\bigr)_{i,j=1}^N\ne0
 \end{array}
 \right\},
\end{equation*}
where
\[
 \ell_n^\mu(q)
 =-\gamma\Lambda_n^\mu\int_Dq(\phi_n^\mu)^2\,d\mu.
\]
For the subsequent disintegration, the corresponding witness event must be
measurable in the environment.  Rational isolating windows provide the needed
scalar encoding.

\begin{lemma}
\label{lem:measurable-submersion}
For every $\mathbf q\in\mathcal Q_N$,
\begin{equation*}
 \mathcal S_{\mathbf q}\in\Borel(\Omega).
\end{equation*}
\end{lemma}

\begin{proof}
Let
\[
 \mathbf a=(a_1,\ldots,a_N),\qquad
 \mathbf c=(c_1,\ldots,c_N)
\]
have rational coordinates with $0<a_i<c_i$, and define the rational-window
event
\begin{equation*}
 \mathcal W_{\mathbf a,\mathbf c}
 :=
 \Omega_*\cap
 \bigcap_{i=1}^N
 \{\Lambda_{n_i-1}^b<a_i<\Lambda_{n_i}^b
 <c_i<\Lambda_{n_i+1}^b\}.
\end{equation*}
The eigenvalue maps in \cref{prop:borel-spectrum} give
$\mathcal W_{\mathbf a,\mathbf c}\in\Borel(\Omega)$.  On this event
\[
 \Pi_i^b:=\mathbf1_{(a_i,c_i)}(A_b)
\]
has rank one.  If $\phi_{n_i}^b$ is any normalized real generator of its
range, then
\begin{equation*}
 \Tr(\Pi_i^b\Mult_{q_j}\Pi_i^b)
 =
 \int_Dq_j(\phi_{n_i}^b)^2\,dM_b.
\end{equation*}
Define on all of $\Omega$
\begin{equation*}
 J_{\mathbf a,\mathbf c,\mathbf q}(b)
 :=
 \begin{cases}
 \det\left(
 -\gamma\Lambda_{n_i}^b
 \Tr(\Pi_i^b\Mult_{q_j}\Pi_i^b)
 \right)_{i,j=1}^N,
 &b\in\mathcal W_{\mathbf a,\mathbf c},\\[1mm]
 0,&b\notin\mathcal W_{\mathbf a,\mathbf c}.
 \end{cases}
\end{equation*}
By \cref{prop:borel-spectrum},
$J_{\mathbf a,\mathbf c,\mathbf q}:\Omega\to\R$ is measurable.
Conversely, every selected eigenvalue of multiplicity one admits rational endpoints
strictly between its neighboring eigenvalues.  Therefore
\begin{equation*}
 \mathcal S_{\mathbf q}
 =
 \bigcup_{\substack{\mathbf a,\mathbf c\in\Q^N\\0<a_i<c_i}}
 \left(
 \mathcal W_{\mathbf a,\mathbf c}
 \cap\{J_{\mathbf a,\mathbf c,\mathbf q}\ne0\}
 \right).
\end{equation*}
The right-hand side is a countable union of sets in $\Borel(\Omega)$, and
therefore $\mathcal S_{\mathbf q}\in\Borel(\Omega)$.
\end{proof}

By the almost-sure simplicity result, \cref{prop:transversality} and
\cref{lem:finite-witness},
\begin{equation}
 \Prob\left(
 \bigcup_{\mathbf q\in\mathcal Q_N}\mathcal S_{\mathbf q}
 \right)=1.
 \label{eq:submersion-cover}
\end{equation}
We next pass to the conditional slice associated with a fixed witness tuple.

\begin{lemma}
\label{lem:fiber-null}
Let $E\in\Borel(\R^N)$ satisfy $\mathcal L^N(E)=0$.  Then, for every
$\mathbf q\in\mathcal Q_N$,
\begin{equation}
 \Prob\bigl(
 \mathcal S_{\mathbf q}
 \cap\{\boldsymbol\Lambda\in E\}
 \bigr)=0.
 \label{eq:fixed-tuple-null}
\end{equation}
\end{lemma}

\begin{proof}
By \cref{lem:measurable-submersion,prop:borel-spectrum}, the event in
\eqref{eq:fixed-tuple-null} belongs to $\Borel(\Omega)$.  Fix
$\mathbf q=(q_1,\ldots,q_N)$ and use the corresponding $N$-dimensional
disintegration from \cref{prop:slices}.  Write its data as
$Y_{\mathbf q},\nu_{\mathbf q},\Phi_{\mathbf q},K_{\mathbf q}$, with
\begin{equation*}
 K_{\mathbf q}(\eta,dt)=k_{\mathbf q}(\eta,t)\,dt.
\end{equation*}
For $\nu_{\mathbf q}$-almost every $\eta$, let
$(\widetilde M_{\eta,t})_{t\in\R^N}$ and $\mathcal G_\eta$ be the coherent
structural family from \cref{cor:coherent-slices}.  Define
\begin{equation}
 \boldsymbol\Lambda_\eta(t)
 :=
 \left(
 \Lambda_{n_1}^{\widetilde M_{\eta,t}},
 \ldots,
 \Lambda_{n_N}^{\widetilde M_{\eta,t}}
 \right).
 \label{eq:fiber-eigen-map}
\end{equation}
By \cref{prop:response}(i), $\boldsymbol\Lambda_\eta$ is continuous on
$\R^N$.  On the
locus where the selected eigenvalues have multiplicity one,
\cref{cor:simple-response} gives local real analyticity.  Put
\begin{equation*}
 Z_\eta
 :=\left\{t\in\R^N:
 \begin{array}{l}
 \Phi_{\mathbf q}(\eta,t)\in\mathcal S_{\mathbf q},\\
 \boldsymbol\Lambda(\Phi_{\mathbf q}(\eta,t))\in E
 \end{array}
 \right\}.
\end{equation*}
Since $\Phi_{\mathbf q}(\eta,\cdot)$ and $\boldsymbol\Lambda$ are
measurable and $\mathcal S_{\mathbf q}\in\Borel(\Omega)$,
$Z_\eta\in\Borel(\R^N)$.  Moreover,
$\mathcal S_{\mathbf q}\subset\Omega_*$, so
$Z_\eta\subset\mathcal G_\eta$.

Fix $t_0\in Z_\eta$.  Then
\[
 \widetilde M_{\eta,t_0}
 =M_{\Phi_{\mathbf q}(\eta,t_0)}.
\]
Choose normalized real eigenfunctions
$\phi_{n_i}^{\eta,t_0}$ of $A_{\widetilde M_{\eta,t_0}}$ for the selected
eigenvalues of multiplicity one.  By \cref{cor:simple-response}, the Jacobian of
\eqref{eq:fiber-eigen-map} is
\begin{equation*}
 \bigl[D\boldsymbol\Lambda_\eta(t_0)\bigr]_{ij}
 =
 -\gamma\Lambda_{n_i}^{\widetilde M_{\eta,t_0}}
 \int_Dq_j(\phi_{n_i}^{\eta,t_0})^2
 \,d\widetilde M_{\eta,t_0}.
\end{equation*}
By the definition of $\mathcal S_{\mathbf q}$,
\begin{equation*}
 \det D\boldsymbol\Lambda_\eta(t_0)\ne0.
\end{equation*}
The inverse function theorem gives an open neighborhood on which
$\boldsymbol\Lambda_\eta$ is a $C^1$ diffeomorphism.  Choose an open
$U_{t_0}$ whose closure is compactly contained in that neighborhood.  Then
\begin{equation*}
 \sup_{y\in\boldsymbol\Lambda_\eta(U_{t_0})}
 \left|
 \det D\!\left[(\boldsymbol\Lambda_\eta|_{U_{t_0}})^{-1}\right](y)
 \right|<\infty.
\end{equation*}
The change-of-variables theorem therefore gives
\begin{align}
 \mathcal L^N\bigl(
 U_{t_0}\cap\boldsymbol\Lambda_\eta^{-1}(E)
 \bigr)
 &=\int_{E\cap\boldsymbol\Lambda_\eta(U_{t_0})}
 \left|
 \det D\!\left[(\boldsymbol\Lambda_\eta|_{U_{t_0}})^{-1}\right](y)
 \right|\,dy\notag\\
 &=0.
 \label{eq:local-null-preimage}
\end{align}

The neighborhoods $U_{t_0}$, with $t_0$ running over $Z_\eta$, cover
$Z_\eta$.  Since $\R^N$ is second countable,
a countable subfamily
still covers $Z_\eta$.  Summing \eqref{eq:local-null-preimage} over this
countable subcover gives
\begin{equation*}
 \mathcal L^N(Z_\eta)=0.
\end{equation*}
Because
$K_{\mathbf q}(\eta,\cdot)\ll\mathcal L^N$,
\begin{equation*}
 K_{\mathbf q}(\eta,Z_\eta)=0.
\end{equation*}
Therefore
\begin{equation*}
 \Prob\bigl(
 \mathcal S_{\mathbf q}\cap\{\boldsymbol\Lambda\in E\}
 \bigr)
 =\int_{Y_{\mathbf q}}
 K_{\mathbf q}(\eta,Z_\eta)\,
 \nu_{\mathbf q}(d\eta)=0,
\end{equation*}
which proves \eqref{eq:fixed-tuple-null}.
\end{proof}

\subsection{Proof of the joint-density assertion}

Let $E\in\Borel(\R^N)$ with $\mathcal L^N(E)=0$.  By
\eqref{eq:submersion-cover} and \cref{lem:fiber-null},
\begin{align*}
 \Prob\{\boldsymbol\Lambda\in E\}
 &\le
 \Prob\left(
 \left(\bigcup_{\mathbf q\in\mathcal Q_N}
 \mathcal S_{\mathbf q}\right)^c
 \right)\\
 &\quad+
 \sum_{\mathbf q\in\mathcal Q_N}
 \Prob\bigl(
 \mathcal S_{\mathbf q}
 \cap\{\boldsymbol\Lambda\in E\}
 \bigr)
 =0.
\end{align*}
Thus
\[
 (\Lambda_{n_1}^b,\ldots,\Lambda_{n_N}^b)_\#\Prob
 \ll\mathcal L^N.
\]
This proves the joint-density assertion and completes the proof of
\cref{thm:main}.

\end{document}